\documentclass[11pt,a4paper]{amsart}

\usepackage[utf8]{inputenc}
\usepackage[T1]{fontenc}
\usepackage{lmodern}
\usepackage[margin=30mm]{geometry}
\usepackage{microtype}
\usepackage{amsmath,amssymb}
\usepackage{graphicx}
\usepackage{textcomp}
\usepackage{listings}
\usepackage{caption}
\usepackage{needspace}
\usepackage{etoolbox}
\usepackage{placeins} % 导言区引入
\usepackage[hidelinks]{hyperref}

\hypersetup{
    pdftitle={A strict amplitude lower bound for the van der Pol equation and extensions to symmetric Lienard systems},
    pdfauthor={Changjian Liu and Ziwei Zhuang},
    pdfsubject={Limit cycle amplitude bounds for symmetric Lienard systems},
    pdfkeywords={van der Pol equation, Lienard system, limit cycle, amplitude, Bernstein basis},
    pdfstartview=FitH
}

\lstdefinestyle{wlcert}{
    language=Mathematica,
    basicstyle=\ttfamily\footnotesize,
    keywordstyle=\bfseries,
    commentstyle=\itshape,
    showstringspaces=false,
    columns=fullflexible,
    keepspaces=true,
    upquote=true,
    breaklines=true,
    frame=single,
    framerule=0.3pt,
    numbers=none,
    xleftmargin=1.5em,
    framexleftmargin=0.3em,
    captionpos=t
}
\AtBeginDocument{\numberwithin{lstlisting}{section}}
\BeforeBeginEnvironment{lstlisting}{\Needspace{8\baselineskip}}

\numberwithin{equation}{section}
\theoremstyle{plain}
\newtheorem{thm}{Theorem}[section]
\newtheorem{lem}{Lemma}[section]
\newtheorem{cor}{Corollary}[section]
\newtheorem{prop}{Proposition}[section]
\theoremstyle{remark}
\newtheorem{rmk}{Remark}[section]
\theoremstyle{definition}

\newtheorem{conjecture}{Conjecture}[section]

\makeatletter
\def\paragraph{\@startsection{paragraph}{4}%
  \z@\z@{-\fontdimen2\font}%
  {\normalfont\bfseries}} % 在这里设置字体，如 \bfseries (加粗) 或 \sffamily (无衬线)
\makeatother

\title[Amplitude lower bounds for the van der Pol equation]
{A strict amplitude lower bound for the van der Pol equation
and extensions to symmetric Li\'enard systems}

\author[C. Liu]{Changjian Liu}
\address{School of Mathematics (Zhuhai), Sun Yat-sen University,
Zhuhai, 519086, P. R. China}
\email{liuchangj@mail.sysu.edu.cn}

\author[Z. Zhuang]{Ziwei Zhuang}
\address{School of Mathematics and System Science,
Guangdong Polytechnic Normal University, Guangzhou, 510665, P. R. China}
\email{zhuangzw@gpnu.edu.cn}
\thanks{Corresponding author: Ziwei Zhuang.}

\keywords{Van der Pol equation, Li\'enard system, limit cycle,
amplitude, Bernstein basis}
\date{}

\begin{document}

\begin{abstract}
We prove that the amplitude of the unique limit cycle of the
van der Pol equation
$\ddot{x}+\mu(x^2-1)\dot{x}+x=0$
is strictly greater than $2$ for every $\mu>0$.
The proof combines orbit comparison in the relaxation oscillation
regime, trigonometric comparison curves in the weakly nonlinear
regime, and polynomial energy estimates in the intermediate regime.
The polynomial sign conditions arising in the latter two regimes
are certified using exact rational arithmetic and Bernstein
coefficients.
The orbit comparison method also applies to a class of symmetric
Li\'enard systems.
Under suitable monotonicity assumptions, we establish a sufficient
criterion for the limit cycle amplitude to exceed a critical value
determined by the damping function.
\end{abstract}

\maketitle

\section{Introduction}

% REVISION 2026-09-20: foreground the van der Pol amplitude theorem.
Li\'enard systems provide a classical framework for the study of nonlinear
and self-sustained oscillations. A general Li\'enard system is given by
\begin{align}\label{sys: generalized Lienard}
    \dot{x}=y-F(x),\quad \dot{y}=-g(x),
\end{align}
or equivalently,
\[
    \ddot x+F'(x)\dot x+g(x)=0,
\]
where dots denote derivatives with respect to time $t$, and $F$ and $g$ are sufficiently smooth functions.
Physically, $F'$ is the damping coefficient
and $g$ is the restoring term.
The existence, uniqueness, and distribution of its limit cycles have been
extensively studied, from the classical work of Levinson and Smith
\cite{LevinsonSmith1942} to subsequent developments surveyed by Llibre and
Zhang \cite{LlibreZhang2017}. For polynomial Li\'enard systems, the
conjecture of Lins Neto, de Melo, and Pugh \cite{lins1977on} on the maximal
number of limit cycles has motivated much of this research.

Beyond existence and uniqueness, the amplitude of a limit cycle, defined
as the maximum of $|x|$ along the orbit, quantifies the size of the
oscillation. Explicit amplitude bounds help locate the cycle in the
phase plane. Such estimates have been obtained for symmetric Li\'enard
systems by Yang and Zeng \cite{yang2015upper} and Cao and Liu
\cite{cao2017estimate}, and for more general Li\'enard-type systems by
Turner, McClintock, and Stefanovska \cite{turner2015maximum}.

An important special case of symmetric Li\'enard systems is the van der Pol
equation
\begin{align}\label{sys: vdP in equation form}
    \ddot{x}+\mu(x^2-1)\dot{x}+x=0,\quad \mu>0.
\end{align}
Introduced by van der Pol \cite{vanderpol1926on} in the study of
relaxation oscillations, it provides a representative model for investigating
amplitude bounds. Its unique stable limit cycle and phase-plane geometry
are classical; see Lefschetz
\cite[Chapter~XI, Sections~2--3]{lefschetz1963differential}.
We denote its amplitude by $A_{\text{van}}(\mu)$.

In the weakly nonlinear limit $\mu\to0^+$, the amplitude has the expansion
\begin{align*}
    A_{\text{van}}(\mu)=2+\frac{1}{96}\mu^2
    -\frac{1033}{552960}\mu^4
    +\frac{1019689}{55738368000}\mu^6
    +\mathcal O(\mu^8).
\end{align*}
Recursive constructions are given by Buonomo \cite{buonomo1998periodic}
and by L\'opez and L\'opez-Ruiz \cite{lopez2007approximating}.
Amore, Boyd, and Fern\'andez \cite{amore2018high} study high-order
expansions and Pad\'e-type resummation.
In the relaxation limit $\mu\to+\infty$, one has
\begin{align*}
    A_{\text{van}}(\mu)=2-\frac{\alpha_{\rm Ai}}{3}\mu^{-4/3}
    +o\left(\mu^{-4/3}\right),
\end{align*}
where $\alpha_{\rm Ai}<0$ is the largest negative zero of the Airy
function $\operatorname{Ai}$ and $-\alpha_{\rm Ai}/3\approx0.779369$;
see Cartwright \cite{cartwright1952van} and Amore et al.\
\cite{amore2018high}.
Both expansions imply $A_{\text{van}}(\mu)>2$ in their respective
asymptotic regimes, but do not control the full intermediate range.

Uniform upper bounds have been established by Odani \cite{odani2000limit}
and in the works on amplitude estimates cited above.
In particular, Yang and Zeng \cite{yang2015upper} obtained $A_{\text{van}}(\mu)<\sqrt{5},$ 
Cao and Liu \cite{cao2017estimate} obtained
$A_{\text{van}}(\mu)<2.0976$, while Turner, McClintock, and Stefanovska
\cite{turner2015maximum} obtained $A_{\text{van}}(\mu)<2.0672$.
Their approaches relate amplitude bounds to positive zeros of
polynomials obtained from recursive constructions.
In particular,
Turner et al.~\cite[Section~III.A and Table~IV]{turner2015maximum}
report progressively sharper uniform upper bounds as the degree of the auxiliary polynomials increases.
For the construction of Cao and Liu, 
the observations of increasingly tight lower and upper bounds at $\mu=1,2$
motivated them to conjecture that, for every fixed $\mu>0$, their lower bounds
increase and their upper bounds decrease to the actual amplitude
$A_{\text{van}}(\mu)$ as the recursion order tends to infinity;
see 
\cite[Tables~1--2 and Conjecture~1]{cao2017estimate}.

By numerical experiments,
Odani \cite{odani2000limit} suggested that
$2<A_{\text{van}}(\mu)<2.0235$ for every $\mu>0$.
The principal result of this paper is a rigorous proof of the strict
lower-bound part of this conjecture throughout the positive parameter range.
\begin{thm}\label{thm: A of vdP > 2}
    The amplitude of the unique limit cycle of the van der Pol equation
    \eqref{sys: vdP in equation form} satisfies
    $A_{\text{van}}(\mu)>2$ for every $\mu>0$.
\end{thm}
Furthermore, numerical computations also suggest that $A_{\text{van}}(\mu)$ first increases and then decreases, with a maximum near $\mu = 3.3$.
If this unimodality is true, 
Theorem \ref{thm: A of vdP > 2} is immediately obtained since 
$A_{\text{van}}(\mu)\rightarrow2$ as $\mu\rightarrow0^+$ or $+\infty$.
However, the conjectured unimodality, as well as the estimate of the upper bound, remains outside the scope of this paper. 
We still have the following cojecture.
\begin{conjecture}
    The amplitude of the unique limit cycle of the van der Pol equation \eqref{sys: vdP in equation form} has a unique critical point for $\mu>0$.
\end{conjecture}

Our proof of Theorem \ref{thm: A of vdP > 2} combines orbit comparison in the relaxation oscillation
regime, trigonometric comparison curves in the weakly nonlinear regime,
and polynomial energy estimates in the intermediate regime.
The latter two constructions lead to finite polynomial sign checks,
performed with exact rational arithmetic in the Bernstein basis.
The intermediate construction builds on the methods of Giacomini and
Neukirch \cite{giacomini1997number,giacomini1998improving}, also used for
amplitude estimates in \cite{turner2015maximum}.

The orbit comparison method for large $\mu$ applies more generally to
symmetric Li\'enard systems. It yields a criterion for the amplitude to
exceed a critical value determined by $F$, together with
an explicit sufficient lower bound on the parameter.
For these general results, we impose the following classical symmetry
and monotonicity hypothesis on system \eqref{sys: generalized Lienard}.
\begin{itemize}
    \item[(H1)] 
    $F(x)$ and $g(x)$
    are odd functions in $C^1(\mathbb R),$
    $g(x)>0$ for $x>0,$
    $F'(0)<0,$
    $F'(x)$ has a unique positive zero at $x=a,$
    and $\lim_{x\rightarrow+\infty}F(x)=+\infty.$
\end{itemize}
Under assumption (H1), we see that
\begin{itemize}
    \item 
    system \eqref{sys: generalized Lienard} is $\mathbb Z_2$-symmetric and has a unique equilibrium point at the origin;
    \item 
    $F(x)$ is $S$-shaped and has a unique positive zero at $x=b>a,$
    and the equation $F(x)=F(-a)$ has a unique positive root at $x= c >b$; see Fig. \ref{fig:lienard_phase_portrait}.
\end{itemize}
The following result is well known for the Li\'enard system under (H1).
\begin{thm}
    Assuming that (H1) holds, we have the following.
    \begin{itemize}
        \item 
        System \eqref{sys: generalized Lienard} has a unique limit cycle.
        
        \item 
        The limit cycle is orbitally asymptotically stable and attracts every non-equilibrium solution.

        \item  
        The amplitude of the limit cycle is greater than $b.$
    \end{itemize}
\end{thm}
For classical formulations and proofs of this result, see
Lefschetz \cite[pp.~267--271]{lefschetz1963differential},
Perko \cite[pp.~253--257]{perko2001differential},
Hartman \cite[pp.~179--181]{hartman1982ordinary},
and Jordan and Smith \cite[pp.~394--399]{jordan2007nonlinear}.
We seek a criterion ensuring that the amplitude of the unique limit cycle of system \eqref{sys: generalized Lienard},
denoted by $A_{\text{Li\'e}},$ is greater than $c.$
We introduce the orbit comparison used to establish this criterion.

The $x$-nullcline
\[
L_F := \left\{(x,y)\;|\;y=F(x),\,x\in\mathbb R\right\}
\]
divides the phase plane into upper and lower regions.
Write the orbit segments of system \eqref{sys: generalized Lienard}
lying above $L_F$ as graphs
\begin{align*}
    y={Y}(x;x_0,y_0),\quad y_0\ge F(x_0) .
\end{align*}
Here $(x_0,y_0)\ne(0,0)$; when $y_0=F(x_0)$, the graph denotes the adjacent orbit segment above $L_F$.
Such a boundary segment leaves the endpoint in forward time for $x_0<0$ and in backward time for $x_0>0$.
By (H1), $Y(x;x_0,y_0)$ is well defined on its interval above $L_F$ and cannot blow up at a finite value of $x$.
It is strictly increasing for $x<0$ and strictly decreasing for $x>0$.
Central symmetry and the attraction property of the unique limit cycle imply that, for $x_0\ge b$,
$A_{\text{Li\'e}}>x_0$ if and only if
\begin{align*}
    {Y}\left(0;-x_0,- F(x_0)\right)>
    {Y}\left(0;x_0, F(x_0)\right).
\end{align*}
In the case $x_0=c$, the bound $A_{\text{Li\'e}}>c$ follows from the stronger condition
\begin{align}\label{ineq:y_left>y_right}
    {Y}\left(0;-a, F( c )\right)\ge
    {Y}\left(0; c , F( c )\right),
\end{align}
because of 
${Y}\left(0;-a, F( c )\right)<{Y}\left(0;- c ,- F( c )\right)$; 
see Fig. \ref{fig:lienard_phase_portrait}.
To establish this orbit comparison, we introduce an auxiliary integral.

\begin{figure}[htbp]
    \centering
    % [width=0.75\textwidth] 控制图片宽度为正文宽度的 75%，可根据页面效果微调
    \includegraphics[width=1\textwidth]{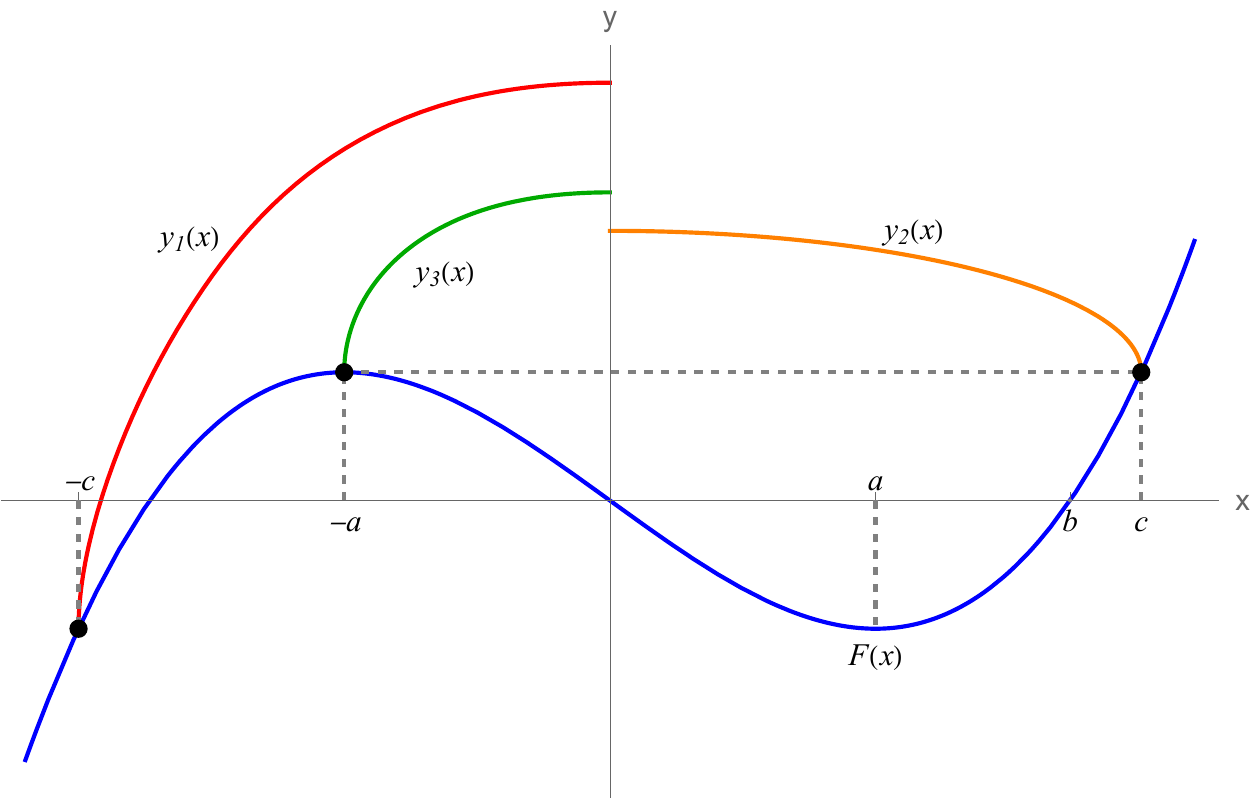}
    
    % 图题说明：支持包含 LaTeX 数学公式
    \caption{The graphs of $F(x),$ $y_1(x):=Y\left(x;-c, -F( c )\right),$ $y_2(x):=Y\left(x;c, F( c )\right)$ and $y_3(x):=Y\left(x;-a, F( c )\right)$ in the case that inequality \eqref{ineq:y_left>y_right} holds.}
    
    % 交叉引用标签（必须放在 \caption 之后）
    \label{fig:lienard_phase_portrait}
\end{figure}

Define
\begin{equation}\label{def: core integral}
    I(v,x,u):=\int_{x}^{u}{\frac{g(s)}{F(s)-F( c )-v} \mathrm d s},\quad 
    v>0 \quad\text{and}\quad x,u< \bar\zeta(v),
\end{equation}
where $\bar\zeta(v)>c$ is the unique real root of the equation $F(x)-F(c)=v>0.$
By (H1), we can see that $F(x)$ is strictly decreasing on $(-a,a)$ and strictly increasing elsewhere;
hence $F(x)<F\left(\bar\zeta(v)\right)=F(c)+v$ for $x< \bar\zeta(v).$ 
Thus, \eqref{def: core integral} is well defined.
The next theorem gives the uniqueness of the positive zero of $I(v,-a,c)$ under an additional hypothesis.
\begin{itemize}
    \item[(H2)] $F'(x)/g(x)$ is nondecreasing on $(0,c)$.
\end{itemize}

\begin{thm}\label{thm: uniqueness of v0}
    Assume that (H1) and (H2) hold.
    The function $v\mapsto I(v,-a,c)$ has a unique positive zero
    at which the sign changes.
    
\end{thm}

This zero helps us establish a criterion for  \eqref{ineq:y_left>y_right}.
We will achieve this by bounding ${Y}\left(0;-a, F( c )\right)$ and 
${Y}\left(0; c , F( c )\right).$

\begin{thm}\label{thm: main theorem 1}
    Assume that (H1) and (H2) hold for system \eqref{sys: generalized Lienard}. 
    Let $v=v_0$ be the positive zero of $I(v,-a, c ).$ 
    If  
    \begin{align}\label{ineq: bound of y-intercepts of orbits}
        \min\left\{{Y}\left(0;-a,F( c )\right),
        {Y}\left(0; c ,F( c )\right)\right\}
        \le v_0 + F( c ),
    \end{align}
    then \eqref{ineq:y_left>y_right} holds.
\end{thm}

From previous analysis, we immediately obtain the following.
\begin{cor}\label{cor: amplitude > c}
    Assume that (H1) and (H2) hold for system \eqref{sys: generalized Lienard}.
    Let $v=v_0$ be the positive zero of $I(v,-a, c ).$ 
    If \eqref{ineq: bound of y-intercepts of orbits} holds,
    then we have $A_{\text{Li\'e}}> c .$
\end{cor}

Inequality \eqref{ineq: bound of y-intercepts of orbits} is particularly useful in the relaxation oscillation regime.
Consider the Li\'enard system with a parameter:
\begin{align}\label{sys: generalized Lienard with mu}
    \dot x = y - \mu F(x),
    \quad 
    \dot y = -g(x),\quad \mu>0,
\end{align}
where $F(x)$ and $g(x)$ still satisfy (H1) and (H2).
% RECHECK 2026-09-09: scaled coordinates and a direct proof of the intercept limits.
To describe the geometry of the relaxation limit, introduce
\[
    z=\frac{y}{\mu},\qquad \tau=\mu t,\qquad \varepsilon=\mu^{-2}.
\]
Then system \eqref{sys: generalized Lienard with mu} becomes
\[
    \frac{\mathrm dx}{\mathrm d\tau}=z-F(x),\qquad
    \frac{\mathrm dz}{\mathrm d\tau}=-\varepsilon g(x).
\]
The corresponding singular loop in the $(x,z)$-plane consists of the horizontal segments
$z=F(c)$ for $x\in[-a,c]$ and $z=-F(c)$ for $x\in[-c,a]$,
together with the branches $z=F(x)$ for $x\in[-c,-a]\cup[a,c]$.
In Fig. \ref{fig:Singular limit cycle}, the vertical coordinate denotes this scaled variable $z=y/\mu$.
This geometry motivates the criterion below, whose proof uses direct orbit estimates under the stated hypotheses.
Writing $Y_\mu$ for the orbit graphs of \eqref{sys: generalized Lienard with mu}, we have
\begin{align*}
    \lim_{\mu\rightarrow+\infty} \left(
    Y_\mu\left(0;-a,\mu F( c )\right)
    -\mu F(c) 
    \right)
    =\lim_{\mu\rightarrow+\infty} \left(
    Y_\mu\left(0;c,\mu F( c )\right)
    -\mu F(c) 
    \right)=0.
\end{align*}
Indeed, for any $\varepsilon>0$ and $x_*\in\{-a,c\}$, comparison with the orbit through the point $(x_*,\mu F(c)+\varepsilon)$ gives
\[
    0<Y_\mu(0;x_*,\mu F(c))-\mu F(c)
    <\varepsilon+\int_{\min\{0,x_*\}}^{\max\{0,x_*\}}
    \frac{|g(s)|}{\mu(F(c)-F(s))+\varepsilon}\,\mathrm ds.
\]
For fixed $\varepsilon$, the integral tends to zero when $\mu\rightarrow+\infty$; letting $\varepsilon\to0^+$ proves the limits.
Thus, the scaled version of \eqref{ineq: bound of y-intercepts of orbits} holds for all sufficiently large $\mu$, since its threshold is $\mu(F(c)+v_0)$.
Below, we write $Y$ for $Y_\mu$ when the dependence on $\mu$
is clear from the context.

\begin{figure}[htbp]
    \centering
    % [width=0.75\textwidth] 控制图片宽度为正文宽度的 75%，可根据页面效果微调
    \includegraphics[width=1\textwidth]{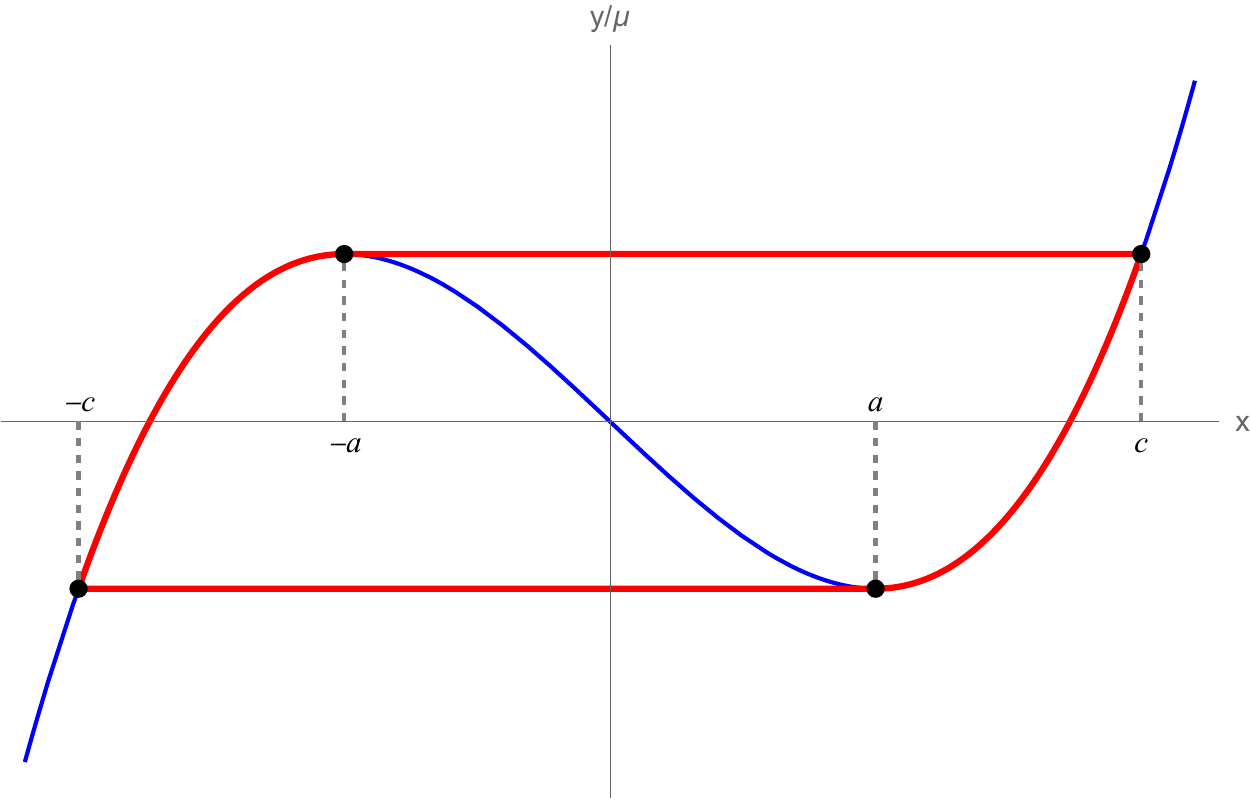}
    
    % 图题说明：支持包含 LaTeX 数学公式
    \caption{The singular loop of system \eqref{sys: generalized Lienard with mu} in the scaled $(x,z)$-plane, where $z=y/\mu$, in the limit $\mu\rightarrow+\infty$.}
    
    % 交叉引用标签（必须放在 \caption 之后）
    \label{fig:Singular limit cycle}
\end{figure}
    
An explicit sufficient lower bound on $\mu$ follows by estimating the orbit intercepts under one additional hypothesis.
\begin{itemize}
    \item[(H3)] $F(x)+F(a)\le -{L} \cdot(x+a)^{l}$
    on $[-a,0]$
    with  some $l>1$ and ${L} >0.$
\end{itemize}

\begin{thm}\label{thm: criterion of mu for A(mu)>b}
    Assume that (H1), (H2) and (H3) hold.
    Let $v=v_0$ be the positive zero of $I(v,-a, c ).$ 
    The amplitude of the limit cycle of system \eqref{sys: generalized Lienard with mu} is greater than $c,$ 
    provided that
    \begin{align*}
        \kappa \left(\frac{1}{\mu}\right)^{\frac{2l}{2l-1}}  \le v_0,
    \end{align*}
    where
    \begin{align}
        \kappa:=
        \left(2-\frac{1}{l}\right)
            \left(\frac{l}{l-1}\right)^\frac{l-1}{2l-1}
            \left( \max_{0\le x \le a}\{g(x)\}\cdot \int_{0}^{+\infty}{\frac{\mathrm d s}
            { s^{l}+1}} \right)^{\frac{l}{2l-1}}
             \left(\frac{1}{ {L} }\right)^{\frac{1}{2l-1}}.
        \label{def: kappa}
    \end{align}
\end{thm}

For the van der Pol equation, we have $F(x)=x^3/3-x$ and $g(x)=x$.
Assumptions (H1)--(H3) hold with $a=1$, $b=\sqrt3$, and $c=2$.
Theorem~\ref{thm: criterion of mu for A(mu)>b} therefore gives the
step for the relaxation oscillation regime in the proof of
Theorem~\ref{thm: A of vdP > 2}.

The paper is organized as follows.
Section~\ref{sec:symmetric} develops the auxiliary integral and orbit
comparisons and proves the general amplitude criteria for symmetric
Li\'enard systems. Section~\ref{sec:vdp} proves
Theorem~\ref{thm: A of vdP > 2} by treating the relaxation oscillation,
weakly nonlinear, and intermediate regimes separately. Appendix~\ref{app:certificates}
provides the complete Mathematica programs and exact computational
certificates used in these proofs.

\section{Results for symmetric Li\'enard systems}\label{sec:symmetric}

In this section, we always assume that (H1) holds.
Recall that $b$ is the unique positive zero of $F(x)$ and $ c $ is the unique positive root of $F(x)=F(-a),$
where $a$ is given in (H1).

First, we study the existence of positive zeros of $I(v,-a, c )$ with respect to $v,$
where $I(v,x,u)$ is defined in \eqref{def: core integral}.
It is sufficient to consider the signs of $I(v,-a, c )$ near $0^+$ and $+\infty,$ respectively.
Since $F$ and $g$ are bounded on $[-a,c]$, we have
\begin{align*}
    I(v,-a, c )\sim -\frac{1}{v}\int_{-a}^{ c }{g(s)\mathrm d s}
    \quad ( v\rightarrow+\infty).
\end{align*}
Moreover, oddness and positivity of $g$ on $(0,+\infty)$ give
$\int_{-a}^c g(s)\,\mathrm ds=\int_a^c g(s)\,\mathrm ds>0$; hence
\begin{align}\label{ineq: I(v,-a, c )<0}
    I(v,-a, c )<0
    \quad\text{for sufficiently large }\; v.
\end{align}
On the other hand, the sign of $I(v,-a, c )$ as $v\rightarrow0^+$ is determined by the sign of the integrand as $s\rightarrow(-a)^+,$
because when $v=0$ the singularity of the integrand at $s=-a$ is stronger than that at $s= c .$
To be specific, by (H1) we have
\begin{align*}
    F(s)+ F(a)=o(s+a)\quad (s\rightarrow(-a)^+)
    \quad
    \\
    \text{and}
    \quad
    F(s)-F( c ) \sim  F'( c )(s- c ) \quad (s\rightarrow  c ),
\end{align*}
and hence
\begin{align*}
    \lim_{v\rightarrow 0^+}
    \frac{I(v,-a,0)}{-\ln v} = +\infty
    \quad\text{and}\quad
    I(v,0, c ) \sim \frac{g( c )}{F'( c )}\ln v \quad (v\rightarrow 0^+),
\end{align*}
respectively.
This implies that
\begin{align}\label{ineq:I(v,-a, c )>0}
    I(v,-a, c )=I(v,-a,0)+I(v,0, c ) >0 
    \quad\text{for $v$ sufficiently close to } 0^+.
\end{align}
Now by \eqref{ineq: I(v,-a, c )<0} and \eqref{ineq:I(v,-a, c )>0},
$I(v,-a, c )$ has at least one positive zero.

Next, we consider the relation between $x$ and $u$ such that $I(v,x,u)=0.$
Recall that $\bar\zeta(v)$ denotes the unique real root of the equation $F(x)-F(c)=v>0.$
When $ x \le 0\le u<  \bar\zeta(v),$
it is not hard to see that
$I(v,x,u)$ is decreasing with respect to both $x$ and $u.$
Together with the fact that 
\begin{align*}
    I(v,x,0)>0>\lim_{u\rightarrow \left(\bar\zeta(v)\right)^-}I(v,x, u ) =-\infty
    \quad\text{for}\quad x<0,
\end{align*}
for any $x\le 0$ there is a unique $u\in\left[0, \bar\zeta(v) \right)$ satisfying $I(v,x,u)=0.$
On the other hand, denote 
\begin{align*}
    \tilde{\zeta}(v):=\sup\left\{u\in(0,\bar\zeta(v))\;|\;
    I(v,x,u)=0\text{ for some }x<0\right\}.
\end{align*}
% RECHECK 2026-09-09: select the nontrivial branch before defining the involution.
We have $\tilde{\zeta}(v)\le\bar\zeta(v)$.
For $x<0$, select the unique positive solution $u$ of $I(v,x,u)=0$.
This strictly decreasing correspondence maps $(-\infty,0)$ onto $(0,\tilde\zeta(v))$.
Extending it by its inverse on $(0,\tilde\zeta(v))$ and assigning $0$ to $0$, 
$I(v,x,u)=0$ defines a continuous involution between $x$ and $u$ on $(-\infty,\tilde\zeta(v))$.
In particular, letting $v=v_0$ be any positive zero of $I(v,-a,c)$,
we denote this involution mapping with $v=v_0$ by
\begin{align*}
    u=\sigma_I(x),\quad x< \tilde{\zeta}(v_0).
\end{align*}
We emphasize some basic properties of $\sigma_I(x)$ as follows.
\begin{itemize}
    \item 
    $I(v_0,x,\sigma_I(x))\equiv0.$

    \item 
    $\sigma_I(x)$ is strictly decreasing .

    \item 
    $\sigma_I(0)=0$ and $\tilde{\zeta}(v_0)>c.$
\end{itemize}

After introducing the following lemma,
we will give our key result (Theorem \ref{thm:F>sigma'F(sigma)}) for Theorem \ref{thm: uniqueness of v0} and Theorem \ref{thm: main theorem 1}.

% RECHECK 2026-09-09: allow a zero interval without strengthening (H2).
\begin{lem}\label{lem: monotonicity of I}
    Assume that (H1) and (H2) hold. 
    For $x\in[-a,0]$, let $\sigma_F(x)\in[b,c]$ be the unique point satisfying $F(\sigma_F(x))=F(x)$.
    For each fixed $v>0$, the function $x\mapsto I(v,x,\sigma_F(x))$ has a possibly degenerate interval of minima contained in $(-a,0)$.
    It is strictly decreasing before this interval and strictly increasing after it.
    \begin{proof}
        By (H1), it is not hard to see that $\sigma_F(x)$ is strictly decreasing and $\sigma_F(0)>a.$
        By direct computation, for $x\in(-a,0)$ we have
        \begin{align*}
            \frac{\mathrm d}{\mathrm d x}I\left(v,x,\sigma_F(x)\right)
            =&\frac{\sigma_F'(x)\cdot g\left(\sigma_F(x)\right)}{F\left(\sigma_F(x)\right)-F( c )-v}
            -\frac{ g(x)}{F(x)-F( c )-v}
            \\
            =&\frac{F'(x)}{F'\left(\sigma_F(x)\right)}\frac{g\left(\sigma_F(x)\right)}{F(x)-F( c )-v}
            -\frac{ g(x)}{F(x)-F( c )-v}
            \\
            =&\left(\frac{g\left(\sigma_F(x)\right)}{F'\left(\sigma_F(x)\right)}-\frac{g(x)}{F'(x)}\right)
            \frac{F'(x)}{F(x)-F( c )-v}.
        \end{align*}
        By (H1), $F'(x)<0$ for $x\in(-a,a)$, and $F(x)-F(c)-v<F(x)-F(c)\le0$ for $x\in[-a,c]$.
        Therefore, the sign of $\mathrm d I\left(v,x,\sigma_F(x)\right)/\mathrm d x$
        is determined by the sign of 
        \[
        \Psi(x):=\frac{g\left(\sigma_F(x)\right)}{F'\left(\sigma_F(x)\right)}-\frac{g(x)}{F'(x)}.
        \]
        On the other hand, 
        by (H1), $g(x)/F'(x)$ is an odd function. 
        Then by (H2) and the fact that $\sigma_F(x)$ is decreasing, 
        $\Psi(x)$ is nondecreasing for $x\in(-a,0)$.
        Its endpoint values satisfy
        \begin{align*}
            \lim_{x\rightarrow (-a)^+}\Psi(x)
        =\frac{g( c )}{F'( c )}-\lim_{x\rightarrow (-a)^+}\frac{g(-a)}{F'(x)}=-\infty
        \quad
        \text{and}\quad \Psi(0)
        =\frac{g(\sigma_F(0))}{F'(\sigma_F(0))}>0,
        \end{align*}
        Thus, the zero set of $\Psi$ is a nonempty closed interval $[x_-,x_+]\subset(-a,0)$, possibly with $x_-=x_+$,
        and  $I(v,x,\sigma_F(x))$ is strictly decreasing on $(-a,x_-)$, constant on $[x_-,x_+]$, and strictly increasing on $(x_+,0)$.
    \end{proof}
\end{lem}

For convenience, let
\begin{align}\label{def: mathcal F}
    \mathcal{F}(x,v):=\frac{g(x)}{F(x)-F( c )-v}.
\end{align}
Recall that 
$I(v,x,u)=\int_x^u {\mathcal{F}(s,v)\mathrm d s}.$

\begin{thm}\label{thm:F>sigma'F(sigma)}
    Assume that (H1) and (H2) hold.
    For any $(x,v)\in(-a,0)\times(0,v_0),$
    we have 
    \begin{align*}
        \mathcal{F}(x,v)>\sigma'_I(x) \mathcal{F}\left(\sigma_I(x),v\right),
    \end{align*}
    while for any $(x,v)\in(-a,0)\times(v_0,+\infty),$
    we have
    \begin{align*}
        \mathcal{F}(x,v)<\sigma'_I(x) \mathcal{F}\left(\sigma_I(x),v\right).
    \end{align*}
    \begin{proof}
        We first claim that $F(\sigma_I(x))< F(x)$
        on $(-a,0).$
        Let 
        \[
        u=\sigma_F(x),\quad x\in[-a,0]
        \]
        be the map defined in Lemma \ref{lem: monotonicity of I}.
        Since $I(v_0,-a,c)=0$ and $I(v_0,0,b)<0$, Lemma \ref{lem: monotonicity of I} gives, for $x\in(-a,0)$,
        \begin{align*}
        I\left(v_0,x,\sigma_F(x)\right)
        &<\max\left\{I\left(v_0,-a,\sigma_F(-a)\right),I\left(v_0,0,\sigma_F(0)\right)\right\}
        \\
        &=0=I\left(v_0,x,\sigma_I(x)\right).
        \end{align*}
        Since $I(v_0,x,u)$ is decreasing with respect to $u,$
        it implies that, for $x\in(-a,0)$, we have
        \[
        0<\sigma_I(x)<\sigma_F(x)< c .
        \]
        Recall that by (H1), $F(x)$ is 
        strictly decreasing on $(-a,a)$ and
        strictly increasing elsewhere, and $\sigma_F(0)>a.$
        Thus,
        when $0<\sigma_I(x)<\sigma_F(0),$
        we have 
        \[
        F(\sigma_I(x))<F(\sigma_F(0))=F(0)<F(x),
        \]
        while when $\sigma_I(x)\ge \sigma_F(0),$ 
        we also have
        \[
        F(\sigma_I(x))<F(\sigma_F(x))=F(x).
        \]
        Now we have shown that $F(\sigma_I(x))< F(x)$ for $x\in(-a,0).$

        From $I\left(v_0,x,\sigma_I(x)\right)\equiv0,$
        we obtain
        \begin{align*}
            0=\frac{\mathrm{d}}{\mathrm d x}I(v_0,x,\sigma_I(x))
            =\sigma_I^\prime(x) \mathcal F(\sigma_I(x),v_0)-\mathcal F(x,v_0).
        \end{align*}
        Direct computation yields
        \begin{align*}
            &\mathcal{F}\left(x,v\right)
            -\sigma'_I(x)\mathcal{F}\left(\sigma_I(x),v\right)
            \\
            =&\mathcal{F}\left(x,v\right)
            -\frac{\mathcal F(x,v_0)}{\mathcal F\left(\sigma_I(x),v_0\right)}
            \mathcal{F}\left(\sigma_I(x),v\right)
            \\
            =&\frac{g(x)}{F(x)-F( c )-v}
            -\frac{F\left(\sigma_I(x)\right)-F( c )-v_0}{F(x)-F( c )-v_0}
            \frac{g(x)}{F\left(\sigma_I(x)\right)-F( c )-v}
            \\
            =&\frac{(v-v_0)\left(F\left(\sigma_I(x)\right)-F(x)\right)g(x)}{\left(F\left(x\right)-F( c )-v\right)\left(F\left(x\right)-F( c )-v_0\right)\left(F\left(\sigma_I(x)\right)-F( c )-v\right)}.
        \end{align*}
        For $(x,v)\in(-a,0)\times(0,+\infty),$
        the expression above has the same sign as $v_0-v,$
        which completes the proof.
    \end{proof}
\end{thm}

\subsection{Proof of Theorem \ref{thm: uniqueness of v0}}

        The existence of a positive zero $v=v_0$ 
        of $I(v,-a,c)$ is given by \eqref{ineq: I(v,-a, c )<0} and \eqref{ineq:I(v,-a, c )>0}.
        On the other hand, we have
        \begin{align*}
            I(v,-a, c )=I(v,-a,0)+I(v,0, c )
            =\int_{-a}^0 \left(
            \mathcal F(x,v)-\sigma_I^\prime(x) \mathcal F(\sigma_I(x),v)
            \right) \mathrm dx.
        \end{align*}
        By Theorem \ref{thm:F>sigma'F(sigma)},
        $I(v,-a, c )$ has the same sign as $v_0-v,$
        giving the uniqueness of positive zero and a sign change there.

\subsection{Proof of Theorem \ref{thm: main theorem 1}}

    Introducing a new time $\tau$ by
    \begin{align*}
       \mathrm d t=\sigma_I'(x)\,\mathrm d\tau,
    \end{align*}
    system \eqref{sys: generalized Lienard} becomes, with dots now denoting derivatives with respect to $\tau$,
    \begin{align}\label{sys: Lienard eq after variable changes}
        \dot x = y- F\left(\sigma_I(x)\right),
        \quad 
        \dot y 
        =-g(x)\cdot \frac{F\left(\sigma_I(x)\right)-F( c )-v_0}{F(x)-F( c )-v_0}.
    \end{align}
    Since $\sigma_I'(x)<0$ on $\left(-\infty,\tilde{\zeta}(v_0)\right)\backslash\{0\}$, this time reparametrization reverses the direction of the flow.
    Although condition (H1) does not guarantee the existence of $\sigma_I'(x)$ at $x=0,$
    the vector field of \eqref{sys: Lienard eq after variable changes} is well defined and continuous for $(x,y)\in\left(-\infty,\tilde{\zeta}(v_0)\right)\times\mathbb R.$

    Consider the curves:
    \begin{align*}
        \mathcal L&:=\left\{(x,y)\;|\;y=Y\left(x;-a, F(c)\right),\,-a<x<0\right\},
        \\
        \mathcal L'&:=\left\{(x,y)\;|\;y=Y\left(x;c, F(c)\right),\,0<x<c\right\},
        \\
        \mathcal L_\sigma&:=\left\{(x,y)\;|\;y=Y\left(\sigma_I(x);c, F(c)\right),\,-a<x<0\right\}.
    \end{align*}
    We see that $\mathcal L$ and $\mathcal L'$ are segments of orbits of system \eqref{sys: generalized Lienard}.
    By the transformation $x\rightarrow \sigma_I(x),$
    $\mathcal L'$ is mapped to $\mathcal L_\sigma$
    which is a segment of orbit of system \eqref{sys: Lienard eq after variable changes}.
    Moreover, $\mathcal L$ and $\mathcal L_\sigma$ have a common endpoint $(-a,F(c)).$
    Now, we are going to show that  $\mathcal L$ lies above $\mathcal L_\sigma$ under the assumptions of the theorem.

    Notice that
    \begin{align}\label{eq: difference between two vector fields}
    \left.\frac{\mathrm dy }{\mathrm d x}\right|_{\eqref{sys: generalized Lienard}}-
    \left.\frac{\mathrm dy }{\mathrm d x}\right|_{\eqref{sys: Lienard eq after variable changes}}
    =
    \mathcal{F}\left(x,y-F( c )\right)
    -\sigma_I'(x)\mathcal{F}\left(\sigma_I(x),y-F( c )
    \right).
    \end{align}
    By Theorem \ref{thm:F>sigma'F(sigma)}, the expression in \eqref{eq: difference between two vector fields} is positive for any $(x,y)\in \mathcal D,$ where
    \begin{align*}
         \mathcal D:=(-a,0)\times\left( F(c), F(c)+  v_0 \right).
    \end{align*}
    % RECHECK 2026-09-09: establish the endpoint ordering in the regular x(y) chart.
    The apparent singularity at $(-a,F(c))$ comes from using $x$ as the independent variable.
    Both vector fields are $C^1$ near this point and satisfy $\dot y=-g(-a)>0$ there, so their local orbits can be written as regular graphs $x=x(y)$.
    Both graphs enter $x>-a$ as $y$ increases from $F(c)$.
    In this chart, the strict slope inequality in \eqref{eq: difference between two vector fields} is reversed by taking reciprocals: the $x$-slope of $\mathcal L$ is smaller than that of $\mathcal L_\sigma$ at the same point of $\mathcal D$.
    Scalar comparison from the common initial value therefore gives $x_{\mathcal L}(y)<x_{\mathcal L_\sigma}(y)$ locally, which means that $\mathcal L$ lies above $\mathcal L_\sigma$ in a neighborhood of the point $(-a,F(c))$.
    Specifically, 
    both graphs are strictly increasing in $x$, so condition \eqref{ineq: bound of y-intercepts of orbits} gives two cases:
    $\mathcal L$ or $\mathcal L_\sigma$ lies in $\mathcal D.$
    
    If $\mathcal L\subset\mathcal D,$ by \eqref{eq: difference between two vector fields} and Theorem \ref{thm:F>sigma'F(sigma)}, the vector field of \eqref{sys: Lienard eq after variable changes} crosses $\mathcal L$ from above to below.
    This crossing direction preserves the local ordering established above, so $\mathcal L$ lies above $\mathcal L_\sigma$ throughout $-a<x<0$.

    If $\mathcal L_\sigma\subset \mathcal D,$
    by \eqref{eq: difference between two vector fields} and Theorem \ref{thm:F>sigma'F(sigma)}, the vector field of \eqref{sys: generalized Lienard} crosses $\mathcal L_\sigma$ from below to above.
    The same comparison argument shows that $\mathcal L$ lies above $\mathcal L_\sigma$ throughout $-a<x<0$.
    In this case, $\mathcal L$ need not lie entirely in $\mathcal D.$
    
    In conclusion, we obtain
    \begin{align*}
        {Y}\left(0;-a, F( c )\right)
        \ge
        {Y}\left(0; c , F( c )\right).
    \end{align*}

\subsection{Proof of Theorem \ref{thm: criterion of mu for A(mu)>b}}

The proof depends on the estimate of ${Y}\left(0;-a,\mu F( c )\right).$
\begin{lem}\label{lem: estimate of y_left}
    Assume that (H1) and (H3) hold for system \eqref{sys: generalized Lienard with mu}. 
    We have
    \begin{align*}
        {Y}\left(0;-a,\mu F( c )\right)-\mu F( c )
        <\kappa \left(\frac{1}{\mu}\right)^{\frac{1}{2l-1}},
    \end{align*}
    where $\kappa$ is defined in \eqref{def: kappa}.
    \begin{proof}
        For convenience, let $\bar g:=\max_{0\le x \le a}\{g(x)\}.$
        For any $\varepsilon>0$,  comparison with the orbit through the point $(-a,\mu F(c)+\varepsilon)$ gives
        \begin{align*}
            {Y}\left(x;-a,\mu F( c )\right)
            <{Y}\left(x;-a,\mu F( c )+\varepsilon\right)
            :=y_\varepsilon(x),
            \quad -a\le x \le 0.
        \end{align*}
        Note that 
        $y_\varepsilon(x)>y_\varepsilon(-a)=\mu F( c )+\varepsilon=\varepsilon-\mu F(a)$
        for $x\in(-a,0].$
        Hence
        \begin{align*}
            y_\varepsilon(0)
            =&y_\varepsilon(-a)+\int_{-a}^0{\frac{g(x)}{\mu F(x)-y_\varepsilon(x)}\mathrm dx}
            \\
            =&\mu F( c ) + \varepsilon
            +\int_{-a}^0{\frac{g(x)}
            {\mu \left( F(x)+F(a) \right)-\left(y_\varepsilon(x)+\mu F(a)\right)}\mathrm dx}
            \\
            <&\mu F( c ) + \varepsilon
            +\int_{-a}^0{\frac{\bar g}
            { \mu{L} \cdot(x+a)^{l}+\varepsilon}\mathrm dx}
            \\
            <&\mu F( c ) + \varepsilon
            + \bar g \left(\frac{1}{\varepsilon}\right)^{\frac{l-1}{l}}
            \left(\frac{1}{\mu {L} }\right)^{\frac{1}{l}}
            \int_{0}^{+\infty}{\frac{\mathrm d s}
            { s^{l}+1}}
        \end{align*}
        Then, we have
        \begin{align*}
             &{Y}\left(0;-a,\mu F( c )\right)-\mu F( c )
             \le \inf_{\varepsilon>0}
             \left\{\varepsilon
            + \bar g \left(\frac{1}{\varepsilon}\right)^{\frac{l-1}{l}}
            \left(\frac{1}{\mu {L} }\right)^{\frac{1}{l}}
            \int_{0}^{+\infty}{\frac{\mathrm d s}
            { s^{l}+1}}\right\}
            \\
            =&\left(2-\frac{1}{l}\right)
            \left(\frac{l}{l-1}\right)^\frac{l-1}{2l-1}
            \left( \bar g \int_{0}^{+\infty}{\frac{\mathrm d s}
            { s^{l}+1}} \right)^{\frac{l}{2l-1}}
             \left(\frac{1}{\mu {L} }\right)^{\frac{1}{2l-1}}
            =\kappa \left(\frac{1}{\mu }\right)^{\frac{1}{2l-1}}.
        \end{align*}
        The infimum is attained at a positive value of $\varepsilon$, so the strict estimate preceding the minimization gives the claimed strict inequality.
    \end{proof}
\end{lem}

Note that the zero $v=v_0$ of $I(v,-a,c)$ in Theorem \ref{thm: criterion of mu for A(mu)>b} is independent of $\mu.$
By Corollary \ref{cor: amplitude > c}, 
for system \eqref{sys: generalized Lienard with mu},
the amplitude of the limit cycle is greater than $c,$
provided that
\[
{Y}\left(0;-a,\mu F( c )\right)-\mu F( c )\le\mu v_0.
\]
Dividing the estimate in Lemma \ref{lem: estimate of y_left} by $\mu$ gives the stated sufficient condition
$\kappa\mu^{-2l/(2l-1)}\le v_0$, and proves Theorem \ref{thm: criterion of mu for A(mu)>b}.

\begin{rmk}
    The parameter criterion in Theorem \ref{thm: criterion of mu for A(mu)>b} relies on an upper bound for either ${Y}\left(0;-a,\mu F(c)\right)$ or ${Y}\left(0;c,\mu F(c)\right)$.
    Estimating the latter may improve the sufficient parameter threshold, since the singularity of $\mathcal F(x,0)$ at $x=c$ is weaker than that at $x=-a$.
    The estimate of ${Y}\left(0;-a,\mu F(c)\right)$ used here gives a simple explicit criterion sufficient for the subsequent application.
\end{rmk}

\section{Results for the van der Pol equation}\label{sec:vdp}

We consider the van der Pol equation \eqref{sys: vdP in equation form} in the form
\begin{align}\label{sys: vdP in standard form}
    \dot x = y,\quad \dot y= -x - \mu (x^2-1) y,
    \quad \mu>0.
\end{align}
In this section, we prove Theorem \ref{thm: A of vdP > 2}.
The proof treats the relaxation oscillation regime $\mu\ge12$,
the weakly nonlinear regime $0<\mu<1$, and the intermediate regime
$1\le\mu\le12$ separately, as outlined below.

For the relaxation oscillation regime, we apply
Theorem \ref{thm: criterion of mu for A(mu)>b} and verify its
sufficient condition for every $\mu\ge12$.

For the weakly nonlinear regime, we construct a family of smooth simple closed comparison curves,
each passing through $(\pm2,0)$,
such that the vector field points strictly outward except at these two points.
We verify that this construction works for every $0<\mu<1$.

For the intermediate regime $1\le\mu\le12$,
we use a polynomial construction originating in \cite{giacomini1997number,giacomini1998improving} and applied to amplitude estimation in \cite{turner2015maximum}; see also the related method in \cite{cao2017estimate}.
We construct an auxiliary energy function whose derivative along
solutions is positive whenever $0<|x|\le2$ throughout this regime.

For the weakly nonlinear and intermediate regimes,
a key step is to verify that certain high-degree bivariate polynomials
with rational coefficients satisfy the required positivity conditions
on suitable rectangles. We perform these checks in Mathematica
using exact rational arithmetic.
For low-degree polynomials, one possible approach is to combine resultant elimination and Sturm root counts for suitable critical-point equations with separate checks on the boundary. For the high-degree polynomials considered here, such elimination can be computationally expensive. The bidegrees reach $(194,48)$ in the weakly nonlinear regime and $(175,87)$ in the intermediate regime. We verify their signs by converting from the power basis to the Bernstein basis;
see, for instance, Zettler and Garloff \cite{zettler1998robustness}
and Farouki \cite{farouki2012bernstein}.
The conversion and the resulting positivity criterion are described below.

\paragraph{Bernstein basis and positivity on rectangles.}
For a nonnegative integer $n$, write the degree-$n$ power and Bernstein bases as row vectors
\[
    \mathbf x_n=(1,x,\ldots,x^n),\qquad
    \mathbf b_n(x)=(b_{0,n}(x),\ldots,b_{n,n}(x)),
\]
where
\[
    b_{j,n}(x):=\binom nj x^j(1-x)^{n-j},\qquad 0\le j\le n.
\]
The change of basis is $\mathbf x_n=\mathbf b_n(x)\,M_n$, where
\begin{equation*}
M_n:=
\begin{pmatrix}
1 & 0 & 0 & \cdots & 0 & \cdots & 0\\[3pt]
1 & \dfrac{1}{\binom n1} & 0 & \cdots & 0 & \cdots & 0\\[12pt]
1 & \dfrac{\binom21}{\binom n1} & \dfrac{1}{\binom n2} & \cdots & 0 & \cdots & 0\\[8pt]
\vdots & \vdots & \vdots & \ddots & \vdots & & \vdots\\[3pt]
1 & \dfrac{\binom i1}{\binom n1} & \dfrac{\binom i2}{\binom n2} & \cdots & \dfrac{1}{\binom ni} & \cdots & 0\\[8pt]
\vdots & \vdots & \vdots & & \vdots & \ddots & \vdots\\[3pt]
1 & \dfrac{\binom n1}{\binom n1} & \dfrac{\binom n2}{\binom n2} & \cdots & \dfrac{\binom ni}{\binom ni} & \cdots & \dfrac{1}{\binom nn}
\end{pmatrix}.
\end{equation*}
Rows and columns are indexed from $0$. 
Equivalently,
$(M_n)_{ij}=\binom ij/\binom nj$ for $j\le i$ and zero otherwise.
In particular, $M_n$ is invertible.

For a polynomial $f(x,y)=\sum_{i=0}^n\sum_{j=0}^m c_{ij}x^iy^j$, 
let
$C_{n,m}(f)=(c_{ij})$ be its power coefficient matrix, retaining zero coefficients up to the chosen degree bounds $(n,m)$.
Its Bernstein coefficient matrix on $[0,1]^2$ is defined as
\begin{equation}\label{def: Bernstein matrix on [0,1]^2}
    B_{n,m}(f):=M_n\,C_{n,m}(f)\,M_m^{\mathsf T},
\end{equation}
so that $f(x,y)=\mathbf b_n(x)\,B_{n,m}(f)\,\mathbf b_m(y)^{\mathsf T}$.
In general, we use its actual bidegree as the degree bounds $(n,m),$
i.e.,
$\deg_{x}f=n$ and $\deg_{y}f=m$.

For a nondegenerate rectangle
$J=[a_1,b_1]\times[a_2,b_2]$, where $a_1<b_1$ and $a_2<b_2$, introduce the local coordinates
\[
    x=a_1+(b_1-a_1)s,\qquad y=a_2+(b_2-a_2)t,
    \qquad (s,t)\in[0,1]^2.
\]
These indexed endpoints are independent of the critical points $a,b,c$ in (H1).
Writing $\Delta_1=b_1-a_1$, the affine substitution matrix in the first variable is
\begin{equation*}\label{eq: affine power matrix}
A_n^{\mathrm{aff}}(a_1,b_1):=
\begin{pmatrix}
1 & a_1 & a_1^2 & \cdots & a_1^n\\[3pt]
0 & \Delta_1 & 2a_1\Delta_1 & \cdots & n a_1^{n-1}\Delta_1\\[3pt]
0 & 0 & \Delta_1^2 & \cdots & \binom n2 a_1^{n-2}\Delta_1^2\\[3pt]
\vdots & \vdots & \vdots & \ddots & \vdots\\[3pt]
0 & 0 & 0 & \cdots & \Delta_1^n
\end{pmatrix},
\end{equation*}
whose entries are
\[
    \bigl(A_n^{\mathrm{aff}}(a_1,b_1)\bigr)_{jk}
    =\begin{cases}
        \displaystyle\binom kj a_1^{k-j}(b_1-a_1)^j,&j\le k,\\
        0,&j>k.
    \end{cases}
\]
Here $a_1^0=1$, also when $a_1=0$, and 
$M_0=A_0^{\mathrm{aff}}=(1)$.
Define 
\begin{align}\label{def: K_n}
K_n(a_1,b_1):=M_n\,A_n^{\mathrm{aff}}(a_1,b_1),
\end{align}
and define the corresponding matrices in the second variable in the same way.
Thus, the Bernstein coefficient matrix on $J$ is defined as
\begin{equation}\label{def: bernstein matrix on J}
    B_{n,m}^J(f):=K_n(a_1,b_1)\,C_{n,m}(f)\,K_m(a_2,b_2)^{\mathsf T}.
\end{equation}
In fact, expanding each power of $a_1+(b_1-a_1)s$ and $a_2+(b_2-a_2)t$ gives
\[
    f\bigl(a_1+(b_1-a_1)s,a_2+(b_2-a_2)t\bigr)
    =\mathbf b_n(s)\,B_{n,m}^J(f)\,\mathbf b_m(t)^{\mathsf T}.
\]

Since the products $b_{i,n}(s)b_{j,m}(t)$ are nonnegative and sum to one on $[0,1]^2$,
we have
\begin{equation}\label{ineq: bernstein rectangle bounds}
    \min_{i,j}\bigl(B_{n,m}^J(f)\bigr)_{ij}
    \le f(x,y)\le
    \max_{i,j}\bigl(B_{n,m}^J(f)\bigr)_{ij},\qquad (x,y)\in J.
\end{equation}
Thus, strictly positive entries of $B_{n,m}^J(f)$ give $f>0$ on the closed rectangle $J$,
and
nonnegative entries of $B_{n,m}^J(f)$, at least one of which is positive, give $f>0$ in the interior of $J$.
Moreover,
when the coefficients of $f$ and the endpoints of $J$ are rational, every entry of $B_{n,m}^J(f)$ is rational.
In this case, sign conditions can be checked using exact arithmetic.
The Mathematica implementation of \eqref{def: bernstein matrix on J} is given in Appendix~\ref{app:bernstein}.

\subsection{Relaxation oscillation regime}
Consider system \eqref{sys: generalized Lienard with mu} with 
\begin{align*}
    F(x)=\frac{x^3}{3}-x,\quad g(x)=x.
\end{align*}
To apply Theorem \ref{thm: criterion of mu for A(mu)>b},
we verify hypotheses (H1)--(H3).
It is easy to see that (H1) holds with $a=1$ and $c=2.$
Next, since $F'(x)/g(x)=x-1/x$ is increasing on $\mathbb R^+,$
(H2) also holds.
Finally, for $x\in[-1,0],$ we have 
\[
F(x)+F(1)=\frac{1}{3}(x-2)(x+1)^2 \le -\frac{2}{3}(x+1)^2,
\]
and thus (H3) holds with $L=2/3$ and $l=2$.
Direct computation gives $\kappa = 3\sqrt[3]{6\pi^2}/4$ in \eqref{def: kappa}.

By a numerical computation,
the positive zero $v=v_0$ of $I(v,-1,2)$ is approximately $0.1158.$
By Theorem \ref{thm: criterion of mu for A(mu)>b},
$A_{\text{van}}(\mu)>2$ for $\mu\ge\mu_*,$
where $\mu_*$ is defined by
$\kappa (\mu_*)^{-\frac{4}{3}}=v_0$
and is approximately $11.262.$
To ensure mathematical rigor and facilitate subsequent algebraic manipulations, 
we establish the integer threshold $\mu=12$ by an exact certificate.

\begin{prop}\label{prop: if mu >= 12, A_van > 2}
    If $\mu\ge12$, then $A_{\text{van}}(\mu)>2$.
\end{prop}

\begin{proof}
    Let $v_0$ be the unique positive zero of $I(v,-1,2)$,
    and set $\bar v=11/100$.
    For $\kappa = 3\sqrt[3]{6\pi^2}/4$
    and $\mu\ge12,$
    we have
    \[
        \bigl(\kappa\,\mu^{-\frac{4}{3}}\bigr)^3
        \le
        \bigl(\kappa\cdot 12^{-\frac{4}{3}}\bigr)^3
        =\frac{\pi^2}{8192}
        <\frac{(22/7)^2}{8192}
        <\left(\frac{11}{100}\right)^3
        =\bar v^3.
    \]
    If $I(\bar v,-1,2)>0$, Theorem \ref{thm: uniqueness of v0} implies that $\bar v<v_0$;
    thus the conclusion of this proposition follows from Theorem \ref{thm: criterion of mu for A(mu)>b}.
    The remaining goal is to show that 
     $I(\bar v,-1,2)>0.$

    For the van der Pol equation in the form \eqref{sys: generalized Lienard with mu}, we have
    \[
        \mathcal F(x,v)=\frac{x}{\frac{1}{3}(x-2)(x+1)^2-v},
    \]
    where $\mathcal F$ is defined in \eqref{def: mathcal F}.
    We expand $ \mathcal F$ about $x=1/2$, using the normalized
    variable $s=(2x-1)/3$:
    \[
         \mathcal F\!\left(\frac{1+3s}{2},v\right)
        =-\frac{4}{9+8v}\,
          \frac{1+3s}{1 + \frac{9}{9+8v} \left(s- s^2- s^3\right)}
        =\sum_{n=0}^{\infty}c_n(v)s^n,
    \]
    where
    \[
        c_{-1}(v)=0,\qquad
        c_0(v)=-\frac{4}{9+8v},\qquad
        c_1(v)=\left(3-\frac{9}{9+8v}\right)c_0(v),
    \]
    and
    \[
        c_n(v)=\frac{9}{9+8v}
        \bigl(-c_{n-1}(v)+c_{n-2}(v)+c_{n-3}(v)\bigr),
        \qquad n\ge2.
    \]
    We see that $c_n(v)\in\mathbb Q$ if $v\in\mathbb Q\backslash\{-9/8\}.$
    
    Recall that, for $v>0,$ $\bar\zeta(v)\,(>2)$ denotes the unique real root of
    $F(x)-F(c)-v=(x-2)(x+1)^2/3-v=0$. The other two roots are
    \[
        \zeta_{\pm}(v)=-\frac{\bar\zeta(v)}{2}
        \pm i\sqrt{\frac{3(\bar\zeta(v))^2}{4}-3}.
    \]
    Direct computation yields
    \[
        \left|\zeta_{\pm}(v)-\frac12\right|^2
        -\left(\bar\zeta(v)-\frac12\right)^2
        =\frac32\left(\bar\zeta(v)-2\right)>0.
    \]
    Therefore, the Taylor series about $x=1/2$ has
    radius $\bar\zeta(v)-1/2>3/2$ and converges absolutely
    and uniformly for $x\in[-1,2]$.

    The partial-fraction decomposition gives
    \begin{align*}
        \mathcal F(x,v)=&
        \sum_{\rho\in\{\bar\zeta(v),\zeta_+(v),\zeta_-(v)\}}
        \frac{\rho}{\rho^2-1}\frac{1}{x-\rho}
        \\
        =&
        -\sum_{\rho\in\{\bar\zeta(v),\zeta_+(v),\zeta_-(v)\}}
        \frac{\rho}{\rho^2-1}
        \sum_{n=0}^{\infty}\frac{(x-1/2)^n}{(\rho-1/2)^{n+1}}
        \\
        =&
        -\sum_{\rho\in\{\bar\zeta(v),\zeta_+(v),\zeta_-(v)\}}
        \frac{\rho}{\rho^2-1}
        \sum_{n=0}^{\infty}
        \left(\frac{3}{2}\right)^n
        \frac{s^n}{(\rho-1/2)^{n+1}}
    \end{align*}
    This implies that
    \[
        c_n(v)=-\left(\frac32\right)^n
        \sum_{\rho\in\{\bar\zeta(v),\zeta_+(v),\zeta_-(v)\}}
        \frac{\rho}
        {(\rho^2-1)(\rho-1/2)^{n+1}}.
    \]

    We now take $v=\bar v$ and write
    $c_n=c_n(\bar v),\,\bar\zeta=\bar\zeta(\bar v),\,\zeta_{\pm}=\zeta_{\pm}(\bar v).$
    Evaluating $x^3-3x-233/100$ at two rational points yields
    \[
        \frac{61}{30}<\bar\zeta<\frac{51}{25},
    \]
    and hence $\bar\zeta-1/2>23/15$,
    \[
        0<\frac{\bar\zeta}{\bar\zeta^2-1}<\frac23,
    \]
    and
    \[
        \left|
        \frac{\zeta_{\pm}}{\zeta_{\pm}^2-1}
        \right|^2
        =\frac{\bar\zeta^2-3}
        {(\bar\zeta^2-1)(\bar\zeta^2-4)}
        <\frac{940896}{341341}<\frac{25}{9}.
    \]
    Thus, the sum of the absolute values of the three
    residues is less than $4$, and the coefficient
    estimate becomes
    \[
        |c_n|\le
        \frac{60}{23}\left(\frac{45}{46}\right)^n,
        \qquad n\ge0.
    \]

    Termwise integration and the symmetry of the interval
    in $s$ give
    \[
        I(\bar v,-1,2)
        =3\sum_{n=0}^{\infty}\frac{c_{2n}}{2n+1}.
    \]
    For an integer $N\ge0$, set
    \[
        S_N:=3\sum_{n=0}^{N}\frac{c_{2n}}{2n+1}.
    \]
    The remainder is bounded explicitly by
    \begin{align*}
        \bigl|I(\bar v,-1,2)-S_N\bigr|
        \le \frac{180}{23}
        \sum_{n=N+1}^{\infty}
        \frac{(45/46)^{2n}}{2n+1}
        <\frac{16560}{91(2N+3)}
        \left(\frac{45}{46}\right)^{2N+2}
        =:\mathcal E_N.
    \end{align*}
    Recall that $c_n\in\mathbb Q$ since $\bar v =11/100\in\mathbb Q.$
    Taking $N=80$, exact rational arithmetic gives
    \[
        I(\bar v,-1,2)
        > S_{80}-\mathcal E_{80}
        >0.
    \]
    The exact rational verification of $S_{80}-\mathcal E_{80}>0$
    is supplied in Appendix~\ref{app:taylor}.
\end{proof}

\subsection{Weakly nonlinear regime}
By the polar coordinate transformation
$x=\rho\cos\theta$ and $y=-\rho\sin\theta$,
system \eqref{sys: vdP in standard form} is changed to
\begin{align*}
    \dot\rho=\mu\rho\sin^2\theta(1-\rho^2\cos^2\theta),
    \quad
    \dot \theta=1+\mu\sin\theta\cos\theta(1-\rho^2\cos^2\theta).
\end{align*}
Setting $r=\rho^2=x^2+y^2$, we obtain
\begin{equation*}\label{sys: vdP in polar}
    \begin{cases}
        \dot r = 2\mu r\sin^2\theta(1-r\cos^2\theta)
    :=P(\theta,r,\mu),
    \\
    \dot \theta = 1+\mu\sin\theta\cos\theta(1-r\cos^2\theta)
    :=Q(\theta,r,\mu).
    \end{cases}
\end{equation*}
We first show that $\dot\theta$ is strictly positive along the limit cycle of system \eqref{sys: vdP in standard form}.

\begin{lem}\label{lem: theta' does not vanish along the lc}
    Let $\Gamma$ denote the orbit of the limit cycle of system \eqref{sys: vdP in standard form}.
    For any $(\theta,r)$ satisfying $(\sqrt{r}\cos\theta,-\sqrt{r}\sin\theta)\in \Gamma,$
    we have $Q(\theta,r,\mu)>0.$
    \begin{proof}
        For $(x,y)=(\sqrt{r}\cos\theta,-\sqrt{r}\sin\theta),$
        we have 
        \begin{align*}
            r\,Q(\theta,r,\mu)=x^2+y^2+\mu xy(x^2-1):=\Delta(x,y).
        \end{align*}
        Consider the total derivative of $\Delta$ along solutions of system \eqref{sys: vdP in standard form} on the curve $\Delta (x,y)=0:$
        \begin{align*}
            \left.\dot{\Delta}\right|_{\{\Delta=0\}} 
            &= \left.
            \bigl[y\,\Delta_x+(-x-\mu(x^2-1)y)\Delta_y\bigr]
            \right|_{\{\Delta=0\}}
            \\
            &=\left. \mu \left[
            y^2(x^2+1)-(x^2-1)(x^2+\mu x y(x^2-1))
            \right]\right|_{\{\Delta=0\}}
            \\
            &=\left. \mu \left[
            y^2(x^2+1)+y^2(x^2-1)
            \right]\right|_{\{\Delta=0\}}
            \\
            &=\left. 2\mu x^2 y^2 \right|_{\{\Delta=0\}}.
        \end{align*}
        Notice that the curves $\Gamma$ and $\Delta=0$ cannot intersect on the coordinate axes.
        If $\{\Delta=0\}\cap\Gamma\neq\emptyset,$ 
        we must have $\left.\dot{\Delta}\right|_{\{\Delta=0\}\cap\Gamma}>0,$
        which is impossible: the periodic function $t\mapsto\Delta(x(t),y(t))$ cannot have only strict upward crossings of zero.
        It follows that $\{\Delta=0\}\cap\Gamma=\emptyset.$
        Since $\Delta=x^2>0$ at the intersections of $\Gamma$ with the $x$-axis,
        continuity gives $\Delta>0$ on $\Gamma$; hence $Q>0$ there.
        
    \end{proof}
\end{lem}

By Lemma \ref{lem: theta' does not vanish along the lc},
the limit cycle of system \eqref{sys: vdP in standard form} can be characterized by the periodic solution of the equation 
$\mathrm d r/\mathrm d\theta = P(\theta,r,\mu)/Q(\theta,r,\mu),$
denoted by
$r=\varphi(\theta,\mu)$
with $\varphi(0,0)=4.$
% RECHECK 2026-09-09: justify local analyticity of the nonzero periodic branch.
It is well known that this periodic branch has a real-analytic continuation to a neighborhood of $\mu=0$.
In this neighborhood, $\varphi(\theta,\mu)$
can be expressed as a power series expansion in $\mu$:
\begin{align}\label{eq: Taylor expansion of the limit cycle in mu}
    \varphi(\theta,\mu)=\sum_{n=0}^{\infty} \varphi_n(\theta)\mu^n.
\end{align}
The coefficients $\varphi_n(\theta)$ are obtained recursively by substituting the series into the polar equation and imposing periodicity; related perturbation constructions are discussed in \cite{buonomo1998periodic}.
Specifically, substituting $r=\varphi(\theta,\mu)$ into
\begin{align*}
    P(\theta,r,\mu) - Q(\theta,r,\mu) \frac{\mathrm d r}{\mathrm d \theta}
\end{align*}
and equating the coefficients of $\mu^n$ yields
\begin{equation}\label{eq: recurrence of phi_n}
\begin{aligned}
    1 :&\; \varphi_0' = 0,
    \\
    \mu :&\; \varphi_1' = 2 \varphi_0 \sin^2\theta ( 1 - \varphi_0 \cos^2\theta ) ,
    \\
    \mu^2 :&\; \varphi_2' = - \varphi_1' \sin\theta \cos\theta ( 1 - \varphi_0 \cos^2\theta ) + 2 \varphi_1 \sin^2\theta ( 1 - 2 \varphi_0 \cos^2\theta ) ,
    \\
    \mu^3 :&\; \varphi_3' = - \varphi_2' \sin\theta \cos\theta (1 - \varphi_0 \cos^2\theta) +2 \varphi_2 \sin^2\theta (1 - 2 \varphi_0 \cos^2\theta) 
    \\
    &\quad~~~ - 2 \varphi_1^2 \sin^2\theta \cos^2\theta + \varphi_1 \varphi_1' \sin\theta \cos^3\theta,
    \\
    \cdots&
    \\
    \mu^n :&\; \varphi_n' = - \varphi_{n-1}' \sin\theta \cos\theta ( 1 - \varphi_0 \cos^2\theta ) + 2 \varphi_{n-1} \sin^2\theta ( 1 - 2 \varphi_0 \cos^2\theta )
    + \tilde{\varphi}_n,
    \\
    \cdots&,
\end{aligned}
\end{equation}
where $\tilde{\varphi}_n$ is an algebraic combination of $\sin\theta$, $\cos\theta$, and the lower-order coefficients $\varphi_j(\theta)$ and their derivatives, with $0\le j\le n-2$.
From the first two equations, 
it follows that $\varphi_0(\theta)$ is constant and is determined by the periodicity of $\varphi_1(\theta).$ 
That is to say, the constant of integration in $\varphi_0(\theta)$ is obtained by solving 
\[
    \int_0^{2\pi}{\varphi_1'(\theta)\mathrm d \theta}
    =\int_0^{2\pi}2\varphi_0\sin^2\theta(1-\varphi_0\cos^2\theta)\,\mathrm d\theta
    =2\pi \varphi_0\left(1-\frac{\varphi_0}{4}\right)=0.
\]
The nonzero branch gives $\varphi_0=4$, and the other solution, $\varphi_0=0$, corresponds to the equilibrium.

Inductively, assume that $\varphi_0(\theta),\ldots,\varphi_{n-2}(\theta)$ have been fully determined, 
and that $\varphi_{n-1}(\theta)$ has also been obtained up to an undetermined constant.
Then, from the equation for the coefficient of $\mu^n$,
the undetermined constant in $\varphi_{n-1}(\theta)$ is obtained by solving $\int_0^{2\pi}{\varphi_n'(\theta)\mathrm d \theta}=0.$
Moreover, the expression of $\varphi_n(\theta)$ is given by
$\varphi_n(\theta)=\int \varphi_{n}'(\theta)\mathrm d \theta,$
yielding a new undetermined constant of integration.

By the aforementioned structure,
it is not hard to verify the following lemma by induction.
\begin{lem}\label{lem: form of varphi_n}
    Each $\varphi_n(\theta)$ in Eq.~\eqref{eq: Taylor expansion of the limit cycle in mu}
    has one of the following forms:
    \begin{align*}
    \varphi_{2n}(\theta) = \sum_{k=0}^{4n} {a}_{n,k} \cos(2k\theta),
    \quad
    \varphi_{2n+1}(\theta) = \sum_{k=1}^{4n+2} {b}_{n,k} \sin(2k\theta) ,
    \quad n=0,1,2,\ldots,
    \end{align*}
where ${a}_{n,k}$ and ${b}_{n,k}$ are rational numbers.
\begin{proof}
    Central symmetry of the limit cycle gives $\pi$-periodicity. Since equation $\mathrm d r/\mathrm d\theta = P/Q$ is invariant under $(\theta,\mu)\mapsto(-\theta,-\mu)$, 
    it implies that
    $\varphi_n(-\theta)=(-1)^n\varphi_n(\theta)$.
    Induction in the recurrence shows that $\varphi_n$ has angular frequencies at most $4n$.
    Integration of the nonconstant Fourier modes and the next periodicity condition determine its coefficients by rational operations, starting from $\varphi_0=4$.
    These observations give the stated forms.
\end{proof}
\end{lem}
We list $\varphi_n(\theta)$ up to $n=3$ as follows
\begin{align*}
    &\varphi_0(\theta)=4,\quad \varphi_1(\theta)=\sin (4 \theta )-2 \sin (2 \theta ),
    \\
    &\varphi_2(\theta)=-\frac{3}{4}  \cos (2 \theta )+\frac{1}{4} \cos (4 \theta )+\frac{5}{12} \cos (6 \theta )-\frac{1}{4} \cos (8 \theta )+\frac{3}{8},
    \\
    &\varphi_3(\theta)=\frac{1}{12} \sin (2 \theta )+\frac{5}{192} \sin (4 \theta )-\frac{23}{72} \sin (6 \theta )+\frac{37}{192} \sin (8 \theta ) +\frac{5}{48} \sin (10 \theta )
    \\
    &~~~~~~~~~~
    -\frac{5}{64} \sin (12 \theta ).
\end{align*}
Appendix~\ref{app:weak} implements this recurrence directly in Mathematica. 
The integration constants are fixed
by the same successive periodicity conditions as above.
Additionally, $A_{\text{van}}(\mu)$ has the expansion
\begin{align*}
    \left( A_{\text{van}}(\mu) \right)^2=\varphi(0,\mu)=\sum_{n=0}^{\infty} \varphi_n(0)\mu^n=\sum_{n=0}^{\infty} \varphi_{2n}(0)\mu^{2n}
\end{align*}
for $|\mu|$ small enough.

\subsubsection{Construction of the comparison curves}
By the transformation 
\begin{align*}
    \mu^2=\frac{\lambda}{1-\lambda}=\lambda+\lambda^2+\lambda^3+\cdots,
    \quad 0\le\lambda<1,
\end{align*}
which was introduced in \cite{buonomo1998periodic},
Eq.~\eqref{eq: Taylor expansion of the limit cycle in mu} can be rewritten, for $|\mu|$ sufficiently small, as an expansion in $\lambda$:
\begin{equation}\label{eq: Taylor expansion of the limit cycle in lambda}
    \begin{aligned}
    \varphi(\theta,\mu)=\sum_{n=0}^{\infty} \varphi_n(\theta)\mu^n
    =&\sum_{n=0}^{\infty} \varphi_{2n}(\theta)\mu^{2n}
    +\mu\sum_{n=0}^{\infty} \varphi_{2n+1}(\theta)\mu^{2n}
    \\
    =&\sum_{n=0}^{\infty} \alpha_n(\theta)\lambda^{n}
    +\mu\sum_{n=0}^{\infty} \beta_n(\theta)\lambda^{n},
    \end{aligned}
\end{equation}
where
\begin{equation}\label{eq: relations between alpha_n beta_n and varphi_n}
    \begin{aligned}
   &\alpha_0(\theta):=\varphi_0(\theta)\equiv4,\quad
   \alpha_n(\theta):=\sum_{k=1}^n\binom{n-1}{k-1}\varphi_{2k}(\theta),\quad n\ge1,
   \\
   \text{and}\quad
   &\beta_0(\theta):=\varphi_1(\theta),
   \quad
   \beta_n(\theta):=\sum_{k=1}^n\binom{n-1}{k-1}\varphi_{2k+1}(\theta),
   \quad n\ge1.
\end{aligned}
\end{equation}
Denote the truncation of Eq.~\eqref{eq: Taylor expansion of the limit cycle in lambda} with respect to a positive integer $n$ by
\begin{equation}\label{def: T_n}
\begin{aligned}
    T_n(\theta,\mu):=&
    \sum_{k=0}^{n} \alpha_k(\theta)\lambda^{k}
    +\mu\sum_{k=0}^{n-1} \beta_k(\theta)\lambda^{k}
    \\
    =&\sum_{k=0}^{n} 
    \alpha_k(\theta)
    \left(\frac{\mu^2}{1+\mu^2}\right)^{k}
    +\mu\sum_{k=0}^{n-1} 
    \beta_k(\theta)
    \left(\frac{\mu^2}{1+\mu^2}\right)^{k}.
\end{aligned}
\end{equation}
We see that $T_n(0,\mu)$ is a truncation of $\left( A_{\text{van}}(\mu) \right)^2$
and $T_n(\theta,0)\equiv4.$
For any positive integer $n,$ a candidate family of comparison graphs is given by
\begin{align}\label{def: R_n}
    r=R_n(\theta,\mu):=
    \frac{T_n(0,0)}{T_n(0,\mu)} 
    T_n(\theta,\mu)
    =
    \frac{4 T_n(\theta,\mu)}{T_n(0,\mu)}.
\end{align}
Obviously, this family is well defined when $T_n(0,\mu)\neq0$ and has the following properties:
\begin{itemize}
    \item 
    $R_n(\theta,\mu)$ is a $\pi$-periodic function with respect to $\theta;$
    \item 
    $R_n(\theta,0)\equiv4$ and $R_n(0,\mu)\equiv4.$
\end{itemize}
Letting
\begin{align}\label{def: definition of V_n}
    V_n(\theta,\mu):=P\left(\theta,R_n(\theta,\mu),\mu\right)-Q\left(\theta,R_n(\theta,\mu),\mu\right)
    \frac{\partial R_n(\theta,\mu)}{\partial \theta},
\end{align}
we have the following.

\begin{lem}\label{lem: if V>0, A(mu)>2}
    Given a fixed $\mu>0$ and a positive integer $n,$
    if 
    $T_n(0,\mu)>0$  and
    \begin{align}\label{ineq: P - Q r'>0}
        V_n(\theta,\mu)>0
        \quad \text{for}\quad
        \theta\in(0,\pi),
    \end{align}
    then we have $A_{\text{van}}(\mu)>2.$ 
    \begin{proof}
        Clearly, $V_n(\theta,\mu)$ is well defined and $\pi$-periodic for $\theta\in\mathbb R,$
        because of $T_n(0,\mu)>0.$
        We distinguish whether $R_n(\theta,\mu)$ takes a nonpositive value or is everywhere positive.
        Recall that by Lemma \ref{lem: theta' does not vanish along the lc} the limit cycle of system \eqref{sys: vdP in standard form} can be written as $r=\varphi(\theta,\mu)$ for $\theta\in\mathbb R,$
        which is also $\pi$-periodic.
       
        Assume that 
        $R_n(\theta_0,\mu)\le0$ for some $\theta_0\in(0,\pi).$
        Consider the curves $r=R_n(\theta,\mu)$ and $r=\varphi(\theta,\mu)$ in the $(\theta,r)$-plane.
        At any intersection with $\theta\in(\theta_0,\pi)$, Lemma \ref{lem: theta' does not vanish along the lc} gives $Q>0$, and
        \[
            \partial_\theta(\varphi-R_n)=\frac{V_n}{Q}>0.
        \]
        Since $\varphi(\theta_0,\mu)>0\ge R_n(\theta_0,\mu)$, a first-crossing argument gives $R_n(\theta,\mu)<\varphi(\theta,\mu)$ for $\theta\in[\theta_0,\pi)$.
        In fact,  $R_n(\pi,\mu)<\varphi(\pi,\mu)$.
        Otherwise, since $Q>0$ along the limit cycle, continuous dependence for the scalar equation $\mathrm dr/\mathrm d\theta=P/Q$ implies that, for sufficiently small $\varepsilon>0$,
        the solution through $\left(\pi,\varphi(\pi,\mu)-\varepsilon\right)$ remains in a neighborhood of the limit cycle where $Q>0$ and lies above $R_n$ at $\theta_0$.
        It must therefore cross $r=R_n(\theta,\mu)$ from above to below at some $\theta_1\in(\theta_0,\pi)$.
        At such a crossing, $V_n(\theta_1,\mu)\le0$, a contradiction.
        Moreover, the comparison principle can now be applied on the interval $(\pi,\theta_0+\pi].$
        By the periodicity,
        it follows that $R_n(\theta,\mu)<\varphi(\theta,\mu)$ for $\theta\in\mathbb R$;
        hence $A_{\text{van}}(\mu)=\sqrt{\varphi(0,\mu)}>\sqrt{R_n(0,\mu)}=2.$

        Assume that $R_n(\theta,\mu)$ is positive for all $\theta\in\mathbb R.$ 
        It follows that
        \begin{align*}
            \mathcal C:=
        \left\{(\sqrt{r}\cos\theta,-\sqrt{r}\sin\theta)\;|\;
         r=R_n(\theta,\mu),\,0\le\theta<2\pi
         \right\}
        \end{align*}
         is a smooth simple closed curve in the $(x,y)$-plane.
         By \eqref{ineq: P - Q r'>0}, the vector field points strictly outward along $\mathcal C$ except possibly at $(\pm2,0)$, which are regular points.
         % RECHECK 2026-09-09: retain the continuous-dependence proof of strict separation.
         By continuity, its outward normal component is nonnegative at these two points as well, so the closed exterior of $\mathcal C$ is forward invariant.
         Since the limit cycle $\Gamma$ attracts every non-equilibrium solution, it lies in this closed exterior and can meet $\mathcal C$ only at $(\pm2,0)$.
         Suppose that $\Gamma$ and $\mathcal C$ are tangent at $(\pm2,0).$
         In this case, the intersection of the exterior of $\mathcal C$ and the interior of $\Gamma$
         consists of two connected components.
         For an orbit passing through a point located in one of the components,
         it has to enter the interior of $\mathcal C$ in positive time,
         thus contradicting the strict outward crossing at every point of $\mathcal{C}$ except $(\pm2,0)$.
         Therefore, $\Gamma$ is disjoint from $\mathcal C$ and lies strictly outside it, which gives $A_{\text{van}}(\mu)>2$.
    \end{proof}
\end{lem}

\begin{rmk}
    For $0<\mu<2$, the assumptions of Lemma \ref{lem: if V>0, A(mu)>2} imply $R_n(\theta,\mu)>0$ for all $\theta$.
    Indeed, at any zero in $(0,\pi)$, we have $Q(\theta,0,\mu)=1+\mu\sin\theta\cos\theta>0$, so $V_n>0$ implies $\partial_\theta R_n<0$.
    Every zero would therefore be a strict downward crossing, which is incompatible with $R_n(0,\mu)=R_n(\pi,\mu)=4$.
    In particular, throughout the range $0<\mu<1$ used below, the comparison curve $\mathcal C$ is well defined.
\end{rmk}

\begin{lem}\label{lem: n=12, P - Q r'>0}
    Taking $n=12,$
    we have $T_n(0,\mu)>0$ on $(0,1)$ and
    $V_n(\theta,\mu)>0$ on $(0,\pi)\times(0,1).$
    \begin{proof}
        % RECHECK 2026-09-09: establish denominator positivity before using the comparison graph.
        Set $D_n(\mu):=(1+\mu^2)^n T_n(0,\mu)$.
        In particular, $D_{12}\in\mathbb Q[\mu^2]$.
        The exact coefficient calculation in Appendix~\ref{app:weak} shows that all coefficients of $D_{12}$ are nonnegative and its constant term is apparently $4$.
        Hence $T_{12}(0,\mu)>0$ for every real $\mu$ and $R_{12}$ is well defined.

        The functions
        \[
        R_n(\theta,\mu),\quad
        P\left(\theta,R_n(\theta,\mu),\mu\right),\quad
        Q\left(\theta,R_n(\theta,\mu),\mu\right)
        \]
        are $\pi$-periodic trigonometric polynomials in $\theta.$
        They can therefore be expressed as sums of terms of the form
        \begin{align*}
            \cos^i \theta  \sin^{2k-i} \theta
            =\frac{\cot^i\theta}{\left(1+\cot^2\theta\right)^k},
            \quad i=0,1,\ldots,2k.
        \end{align*}
        Thus,
        there must be a positive integer $N_1$ such that
        \begin{align*}
            (1+\cot^2\theta)^{N_1}\, V_n(\theta,\mu)
        \end{align*}
        is a polynomial in $\cot\theta$, with coefficients that are rational functions of $\mu$.
        
        On the other hand, Eqs. \eqref{def: T_n} and \eqref{def: R_n} show that
        \[
        D_n(\mu)R_n(\theta,\mu)
        \]
        is a polynomial in $\mu$.
        Since $P$ is quadratic in $r$ and $Q$ is linear in $r$, the expression
        \[
            D_n(\mu)^2 V_n(\theta,\mu)
        \]
        is also a polynomial in $\mu$, with trigonometric polynomial coefficients.
         
        Furthermore, since $R_n(\theta,0)\equiv4,$
        it implies that $V_n(\theta,0)\equiv0.$
        Thus, the above polynomial is divisible by $\mu$.
        Removing the maximal factor $\mu^{N_2}$ yields a polynomial in both $\xi=\cot\theta$ and $\mu$:
        \begin{align}\label{def: nomalization of V_n}
            \tilde V_n(\xi,\mu):=
            (1+\cot^2\theta)^{N_1}\,
            \mu^{-N_2}\,
            D_n(\mu)^2\,V_n(\theta,\mu).
        \end{align}
        Moreover, by the structure of $\tilde V_n(\xi,\mu),$ 
        it is not hard to see that any of its coefficients is a rational number.
        For $n=12,$
        the exact construction gives
        $N_1=98,N_2=3$ and
        $\deg_{\xi}\tilde V_{12}=194,\,\deg_{\mu}\tilde V_{12}=48.$
        A complete Mathematica construction of $\tilde V_{12}$, with this exact normalization, is given in Appendix~\ref{app:weak}.
        
        Now, in order to verify $V_{12}(\theta,\mu)>0$ on $(0,\pi)\times(0,1),$
        it is equivalent to verify 
        $\tilde V_{12}(\xi,\mu)>0$ on $\mathbb R\times(0,1).$ 
        This is completed by four steps as follows.
        \begin{itemize}
            \item [(i)]
            Divide $\mathbb R\times(0,1)$ into five subdomains:
            \begin{align*}
                &J_0:=\{0\}\times(0,1),
                \quad
                J_1:=(0,1)^2,
                \quad
                J_2:=(-1,0)\times(0,1),
                \\
                &J_3:=[1,+\infty)\times(0,1),
                \quad
                J_4:=(-\infty,-1]\times(0,1).
            \end{align*}

            \item [(ii)]
            For $J_0,$
            the exact coefficient calculation in Appendix~\ref{app:weak} gives
            \begin{align*}
                \tilde V_{12}(0,\mu)=\mu^{22}\sum_{j=0}^{12}q_j\mu^{2j},
                \qquad q_j>0.
            \end{align*}
            Hence $\tilde V_{12}(\xi,\mu)>0$ on $J_0.$

            \item [(iii)]
            The identity and a reflection map $J_1$ and $J_2$, respectively, onto $(0,1)^2$.
            For $J_3$ and $J_4$, reciprocal substitutions give $(0,1]\times(0,1)$, which is contained in $[0,1]^2$.
            Define
            \begin{align*}
            &W_1(\xi,\mu):=\tilde V_{12}(\xi,\mu),
            \quad
            W_2(\xi,\mu):=\tilde V_{12}(-\xi,\mu),
            \\
            &W_3(\xi,\mu):= \xi^{194}\, \tilde V_{12}(1/\xi,\mu),
            \quad
            W_4(\xi,\mu):=W_3(-\xi,\mu).
            \end{align*}
            We see that $W_1,W_2,W_3,W_4\in\mathbb Q[\xi,\mu].$

            \item[(iv)]
            Using the Mathematica code in Appendix~\ref{app:weak}, we compute
            the Bernstein coefficient matrices of $W_i$ on $[0,1]^2$,
            as defined in \eqref{def: Bernstein matrix on [0,1]^2},
            in exact rational arithmetic and obtain
            \begin{align*}
            &\min_{i,j}\bigl(B_{194,48}(W_1)\bigr)_{ij}
            =\min_{i,j}\bigl(B_{194,48}(W_2)\bigr)_{ij}=0,
            \\
            \quad
            &\min_{i,j}\bigl(B_{194,48}(W_3)\bigr)_{ij}=\frac{16}{3},
            \qquad
            \min_{i,j}\bigl(B_{194,48}(W_4)\bigr)_{ij}=\frac{3103}{582}.
            \end{align*}
            The Bernstein positivity criterion \eqref{ineq: bernstein rectangle bounds} gives $W_1,W_2>0$ on $(0,1)^2$ and $W_3,W_4>0$ on $[0,1]^2$.
            The substitutions in step~(iii) therefore give $\tilde V_{12}>0$ on each of $J_1,J_2,J_3,J_4$.
        \end{itemize}
        Now, we have shown that $\tilde V_{12}(\xi,\mu)>0$ on $J_0\cup J_1 \cup J_2 \cup J_3 \cup J_4 = \mathbb R \times (0,1).$
        This completes the proof.
    \end{proof}
\end{lem}

\begin{rmk}
    The reason for using the truncation of \eqref{eq: Taylor expansion of the limit cycle in lambda} instead of \eqref{eq: Taylor expansion of the limit cycle in mu} to construct $R_n(\theta,\mu)$ is that
    the transformation improves the usefulness of the finite truncations in our computations.
    It is motivated by the transformed perturbation expansions in \cite{buonomo1998periodic}; see also the high-order series and resummation study in \cite{amore2018high}.
    The time-phase expansion studied there is distinct from the polar-angle graph used here, so its convergence properties do not directly establish convergence of \eqref{eq: Taylor expansion of the limit cycle in lambda}.
    Every finite real $\mu$ gives $\lambda\in[0,1)$, but this fact alone does not prove convergence on $|\lambda|<1$.
    Our proof uses only a finite truncation and its exact sign certificate, and is independent of any global convergence claim.
\end{rmk}

By Lemmas \ref{lem: if V>0, A(mu)>2} and \ref{lem: n=12, P - Q r'>0}, 
we immediately obtain the following.
\begin{prop}\label{prop: if mu<1, A_van > 2}
    If $0<\mu<1,$ then $A_{\text{van}}(\mu)>2.$
\end{prop}

\subsection{Intermediate regime}

Consider system \eqref{sys: generalized Lienard with mu} in the form
\begin{align}\label{sys: generalized Lienard with mu in standard from}
    \dot x = y,\quad 
    \dot y = -g(x) - \mu F'(x)  y .
\end{align}
Here again $y$ denotes the velocity coordinate.
For a positive integer $n,$
define an auxiliary energy function by 
\begin{align*}
    H_n(x,y,\mu):=\sum_{k=0}^{2n} {h_{2n-k,n}(x,\mu)y^k},
\end{align*}
where $h_{k,n}(x,\mu)$ is given recursively by
\begin{equation}\label{def: h_{k,n}}
    \begin{aligned}
    &h_{0,n}(x,\mu):=1,
    \quad
    h_{1,n}(x,\mu):=2n \mu F(x),
    \\
    &h_{k,n}(x,\mu) := \int_0^x \Big[ (2n - k + 1) \mu F'(s) h_{k-1,n}(s,\mu)
    \\
    &\hspace{42mm}{}+ (2n - k + 2) g(s) h_{k-2,n}(s,\mu) \Big]\,\mathrm ds
\end{aligned}
\end{equation}
for $k=2,3,\ldots,2n.$

A direct computation, with primes denoting derivatives with respect to $x$, yields
\begin{align*}
    \left.\dot H_n\right|_{\eqref{sys: generalized Lienard with mu in standard from}}
    =& y\left(\sum_{k=0}^{2n}{h_{2n-k,n}'\,y^k} \right)
    -(g+ \mu\,y\,F') \left( \sum_{k=1}^{2n} {k\,h_{2n-k,n}\,y^{k-1}} \right)
    \\
    =& -g\, h_{2n-1,n} + h_{0,n}'\,y^{2n+1}
    \\
    &+ \sum_{k=1}^{2n-1} {\left(h_{2n-k+1,n}'-k\,\mu\,\,F'\,h_{2n-k,n} -(k+1)\, g\, h_{2n-k-1,n}\right)}y^k
    \\
    &+\left(h'_{1,n}-2n\,\mu\,F'\,h_{0,n}\right)y^{2n}
    \\
    =&-g\, h_{2n-1,n}.
\end{align*}
If $g(x)h_{2n-1,n}(x,\mu)$ has a fixed strict sign for $0<|x|\le\delta$,
integrating $\dot H_n$ over a period shows that the centrally symmetric limit cycle of system \eqref{sys: generalized Lienard with mu in standard from} cannot lie entirely in $[-\delta,\delta]\times\mathbb R$.
Therefore, its amplitude exceeds $\delta$.

\begin{lem}\label{lem: g h_2n-1 neq 0 for n=88}
    For $g(x)=x,\,F(x)=x^3/3-x$ and $n=88,$
    we have $$-g(x) h_{2n-1,n}(x,\mu)>0$$ for $(x,\mu)\in(0,2]\times[1,12].$
    \begin{proof}
        Under the assumptions of the lemma, 
        it is not hard to deduce by induction that
        \begin{itemize}
            \item 
            $h_{k,n}(x,\mu)\in \mathbb Q[x,\mu],$
            \item 
            $h_{k,n}(-x,\mu) = (-1)^k h_{k,n}(x,\mu),$
            \item 
            $h_{k,n}(x,-\mu) = (-1)^{k} h_{k,n}(x,\mu),$
            \item 
            $h_{k,n}(x,\mu)=\mathcal O(x^{k})\quad (x\rightarrow0),$
        \end{itemize}
        for $k=0,1,\ldots, 2n.$
        Set $\chi=x^2$ and $\nu=\mu^2$. By these properties,
        there is a $\tilde H_n(\chi,\nu)\in \mathbb Q[\chi,\nu]$ such that
        \begin{align}\label{eq: tilde H_n}
            \tilde H_n(x^2,\mu^2)=
            -\frac{g(x)\,h_{2n-1,n}(x,\mu)}{x^{2n} \mu}
            =-\frac{\,h_{2n-1,n}(x,\mu)}{x^{2n-1} \mu}.
        \end{align}
        Now, the proof reduces to verifying that
        $\tilde H_{88}(\chi,\nu)$ is strictly positive on $[0,4]\times[1,144],$
        with $\deg_{\chi}\tilde H_{88}=175$
        and $\deg_{\nu}\tilde H_{88}=87.$
        We verify this positivity in two steps.
        \begin{itemize}
            \item [(i)]
            Divide $[0,4]\times[1,144]$ into four subdomains:
            \begin{align*}
                &J_1:=[2,4]\times[145/2,144],
                \quad
                J_2:=[0,2]\times[145/2,144],
                \\
                &J_3:=[0,2]\times[1,145/2],
                \quad
                J_4:=[2,4]\times[1,145/2].
            \end{align*}
            Write 
            $J_i=[\chi^i_{\min},\chi^i_{\max}]\times[\nu^i_{\min},\nu^i_{\max}]$
            for $i=1,2,3,4.$

            \item [(ii)]
            Using the Mathematica code in Appendix~\ref{app:energy}, we compute
            the Bernstein coefficient matrices of $\tilde H_{88}$ on $J_i$,
            as defined in \eqref{def: bernstein matrix on J},
            in exact rational arithmetic and obtain
            \begin{align*}
            &\min_{i,j}\bigl(B_{175,87}^{J_1}(\tilde H_{88})\bigr)_{ij}>0,
            \qquad
            \min_{i,j}\bigl(B_{175,87}^{J_2}(\tilde H_{88})\bigr)_{ij}>0,
            \\
            \quad
            &\min_{i,j}\bigl(B_{175,87}^{J_3}(\tilde H_{88})\bigr)_{ij}>0,
            \qquad
            \min_{i,j}\bigl(B_{175,87}^{J_4}(\tilde H_{88})\bigr)_{ij}>0.
            \end{align*}
            The Bernstein positivity criterion
            \eqref{ineq: bernstein rectangle bounds} therefore gives
            $\tilde H_{88}>0$ on each of the four rectangles.
        \end{itemize}
        Thus, we have $\tilde H_{88}(\chi,\nu)>0$ on
        $J_1 \cup J_2 \cup J_3 \cup J_4=[0,4]\times[1,144].$
        This completes the proof.
    \end{proof}
\end{lem}

\begin{rmk}
    The construction in this subsection uses a fixed index $n$ on a compact parameter interval.
    For the van der Pol equation, numerical experiments suggest that the index $n$ required to keep $-x h_{2n-1,n}(x,\mu)$ nonnegative for $x\in[-2,2]$ increases as $\mu$ approaches $0$ or $+\infty$.
    No assertion about the limiting growth of the required index is needed in the proof.
\end{rmk}

By the previous analysis, Lemma \ref{lem: g h_2n-1 neq 0 for n=88} implies the following.
\begin{prop}\label{prop: if 1 <= mu <= 12, A_van > 2}
    If $1\le \mu \le 12,$ then $A_{\text{van}}(\mu)>2.$
\end{prop}

\subsection{Proof of Theorem \ref{thm: A of vdP > 2}}
Theorem \ref{thm: A of vdP > 2} follows from Propositions \ref{prop: if mu >= 12, A_van > 2}, \ref{prop: if mu<1, A_van > 2}, and \ref{prop: if 1 <= mu <= 12, A_van > 2}, since their parameter ranges cover $(0,+\infty)$.

\section*{Acknowledgments}
This work is supported by National Natural Science Foundation of China (No. 12571192).
The authors would like to thank Dr. Jiahui Luo for providing ChatGPT as a supporting tool during the writing of this manuscript. 
Specifically, AI assistance was restricted to improve writing and provide design suggestions for the computational algorithms and code routines in the Appendix. 
All mathematical derivations, code implementations, and final manuscript contents were independently verified and authored by the authors.

\appendix
\numberwithin{table}{section}
\section{Exact computational certificates}\label{app:certificates}

This appendix supplies the complete Mathematica code for the finite
computations used in Propositions~\ref{prop: if mu >= 12, A_van > 2},
\ref{prop: if mu<1, A_van > 2}, and
\ref{prop: if 1 <= mu <= 12, A_van > 2}.
The constructions are given by finite recurrences, followed by exact rational coefficient and sign checks. 

All listings were evaluated in Mathematica~14.0.0 on
\texttt{Windows-x86-64}. Each assertion succeeded and the resulting
quantities are recorded below. To reproduce the calculations, evaluate
Listings~\ref{lst:common}--\ref{lst:energy-checks} in their displayed order
in a fresh kernel. The listings use only built-in Wolfram Language
functions and require no external data or packages. 
The command \texttt{must} aborts the evaluation if its argument
does not evaluate to \texttt{True}.

\subsection{The Bernstein coefficient matrix}\label{app:bernstein}

The following routines implement \eqref{def: Bernstein matrix on [0,1]^2} and
\eqref{def: bernstein matrix on J}.

The routine \texttt{coefficientMatrix} stores the coefficient $c_{ij}$
of $x^iy^j$ in row $i+1$ and column $j+1$.
Zero coefficients are retained up to the prescribed degree bounds $(n,m)$.
The routine \texttt{bernsteinMap} returns $K_n(a_1,b_1)$, given in \eqref{def: K_n}, for the input
\texttt{bernsteinMap[n,\{a1,b1\}]}.
Its local arrays \texttt{base} and \texttt{shift} represent $M_n$ and
$A_n^{\mathrm{aff}}(a_1,b_1)$, respectively.
The Bernstein positivity criterion is \eqref{ineq: bernstein rectangle bounds}.

\begin{lstlisting}[style=wlcert,caption={Shared routines for exact coefficient and Bernstein calculations.},label={lst:common}]
(* Run this listing first; reset definitions and disable output history. *)
ClearAll["Global`*"];
$HistoryLength = 0;

(* Abort unless an assertion is exactly True. *)
must[test_, tag_String] := If[!TrueQ[test],
  Print["FAILED: ", tag]; Abort[]];
(* Accept exact integers and rationals, excluding approximate numbers. *)
rationalQ[q_] := IntegerQ[q] || Head[q] === Rational;
(* Store the coefficient of vars[[1]]^i vars[[2]]^j at {i+1,j+1};
   pad with zeros to retain the declared bidegree. *)
coefficientMatrix[f_, vars_, degrees_] := Module[{a},
  must[PolynomialQ[f, vars], "polynomial input"];
  must[And @@ MapThread[Exponent[f, #1] <= #2 &, {vars, degrees}],
    "declared bidegree"];
  a = PadRight[CoefficientList[f, vars], degrees + 1];
  must[VectorQ[Flatten[a], rationalQ], "exact coefficients"];
  a
];

bernsteinMap[d_Integer, {a_, b_}] := Module[{base, shift},
  (* Convert power coefficients on [0,1] to degree-d Bernstein form. *)
  base = Table[If[j <= i,
    Binomial[i, j]/Binomial[d, j], 0], {i, 0, d}, {j, 0, d}];
  (* First substitute x = a + (b-a)t; the product is K_d(a,b). *)
  shift = Table[If[j <= k,
    Binomial[k, j] If[j == k, 1, a^(k - j)] (b - a)^j,
    0], {j, 0, d}, {k, 0, d}];
  base . shift
];

(* Return {dimensions, negative count, zero count, minimum entry}. *)
matrixReport[b_] := Module[{v = Flatten[b]},
  must[VectorQ[v, rationalQ], "exact Bernstein entries"];
  {Dimensions[b], Count[v, _?(# < 0 &)], Count[v, 0], Min[v]}
];
(* Find k with 10^k <= q < 10^(k+1) using exact comparisons. *)
lowerPower[q_] := Module[{k},
  must[q > 0 && rationalQ[q], "positive rational minimum"];
  k = IntegerLength[Numerator[q]] - IntegerLength[Denominator[q]];
  If[q < 10^k, k--];
  must[10^k <= q < 10^(k + 1), "power-of-ten enclosure"];
  k
];
Print["Environment: ", {$Version, $SystemID}];
\end{lstlisting}

Each output of \texttt{matrixReport} is an ordered list
$\{\text{dimensions},n_-,n_0,m\}$, where $n_-$ and $n_0$ count
the negative and zero entries, respectively, and $m$ is the minimum entry.
For a positive rational number $q$, \texttt{lowerPower[q]} finds an
integer $k$ such that $10^k\le q<10^{k+1}$.
It uses integer digit counts and exact comparisons, 
so this estimate does not involve a floating-point logarithm.

\subsection{The certificate for the relaxation oscillation regime}\label{app:taylor}

For $\bar v=11/100$, the recurrence in the proof of
Proposition~\ref{prop: if mu >= 12, A_van > 2} determines
$c_{-1},c_0,\ldots,c_{160}$ using rational arithmetic alone.
The following code verifies the auxiliary rational bounds used
in that proof and checks that
$S_{80}-\mathcal E_{80}>0$.
The variables \texttt{s80} and \texttt{d80} represent
$S_{80}$ and the remainder bound $\mathcal E_{80}$, respectively.

\begin{lstlisting}[style=wlcert,caption={Exact verification for the relaxation oscillation regime.},label={lst:taylor}]
(* Use rational parameters and the Taylor variable s = (2x-1)/3. *)
vbar = 11/100;
q = 9/(9 + 8 vbar);
order = 80;

c = ConstantArray[0, 2 order + 2];
(* c[[j + 2]] represents c_j, including c_{-1}. *)
(* Set c_0 and c_1; c_{-1} remains zero. *)
c[[2]] = -4/(9 + 8 vbar);
c[[3]] = (3 - q) c[[2]];
(* Generate the remaining coefficients through c_160. *)
Do[
  c[[j + 2]] =
    q (-c[[j + 1]] + c[[j]] + c[[j - 1]]),
  {j, 2, 2 order}
];

(* Only even powers contribute to the integral over -1 <= s <= 1. *)
s80 = 3 Sum[
  c[[2 j + 2]]/(2 j + 1), {j, 0, order}
];
(* d80 is the proved remainder bound E_80. *)
d80 = 16560/(91 (2 order + 3)) *
  (45/46)^(2 order + 2);

rootPolynomial[t_] := t^3 - 3 t - 233/100;

taylorChecks = {
  (* Endpoint signs locate the real pole between 61/30 and 51/25. *)
  rootPolynomial[61/30] == -629/27000,
  rootPolynomial[51/25] == 2479/62500,
  (* Check the squared-modulus bound for the complex residues. *)
  940896/341341 < 25/9,
  (* Use pi < 22/7 to verify the parameter bound at mu = 12. *)
  (22/7)^2/8192 < vbar^3,
  (* Certify the positive integral lower bound directly. *)
  s80 - d80 > 0
};

must[And @@ taylorChecks, "Taylor certificate"];
Print["Taylor checks: ", taylorChecks];
\end{lstlisting}

The printed list consists of five \texttt{True} values,
the last of which corresponds to
$S_{80}-\mathcal E_{80}>0$ in exact rational arithmetic.

\subsection{The certificate for the weakly nonlinear regime}\label{app:weak}

% REVISION 2026-09-19: direct real trigonometric recurrence and certificate.
\paragraph{The coefficient recurrence and periodicity condition.}
We follow the construction of \eqref{eq: Taylor expansion of the limit cycle in mu}
directly, retaining the sine and cosine representation in
Lemma~\ref{lem: form of varphi_n}. For $j\ge1$, each coefficient equation
in \eqref{eq: recurrence of phi_n} can be written as $\varphi_j'=\mathcal G_j$, where
\begin{align*}
\mathcal G_j:={}&2\sin^2\theta\,\varphi_{j-1}
 -2\sin^2\theta\cos^2\theta
   \sum_{k=0}^{j-1}\varphi_k\varphi_{j-1-k}\notag\\
 &-\sin\theta\cos\theta\,\varphi_{j-1}'
 +\sin\theta\cos^3\theta
   \sum_{k=0}^{j-1}\varphi_k\varphi_{j-1-k}'.
\end{align*}
For a trigonometric polynomial $f$, let
$\langle f\rangle=(2\pi)^{-1}\int_0^{2\pi}f(\theta)\,\mathrm d\theta$.
Periodicity requires $\langle\mathcal G_j\rangle=0$.

Starting from $\varphi_0=4$, suppose that the provisional
$\varphi_{j-1}$ has zero mean and that its integration constant is
still undetermined. For $j\ge2$, adding a constant $\delta$ to
$\varphi_{j-1}$ yields additional terms for $\mathcal G_j$:
\[
 \delta\,2\sin^2\theta(1-8\cos^2\theta)
 \qquad\text{with}\qquad
 \bigl\langle2\sin^2\theta(1-8\cos^2\theta)\bigr\rangle=-1.
\]
Thus, the required constant is $\delta=\langle\mathcal G_j\rangle$. 
Once this constant is fixed, we
integrate the resulting zero-mean right-hand side term by term:
\[
 \int\cos(2k\theta)\,\mathrm d\theta
   =\frac{\sin(2k\theta)}{2k},\qquad
 \int\sin(2k\theta)\,\mathrm d\theta
   =-\frac{\cos(2k\theta)}{2k}, \quad k\ge1.
\]
The next periodicity condition gives the new integration constant.
In particular, constructing $T_{12}$ requires
$\varphi_0,\ldots,\varphi_{24}$, so the loop below runs through
$j=25$ to determine the constant in $\varphi_{24}$.
The provisional $\varphi_{25}$ is not used subsequently.

In the code, \texttt{phi[j]} represents $\varphi_j(\theta)$.
The routine \texttt{trigForm} reduces products to a linear combination
of harmonics; \texttt{trigMean} and \texttt{trigPrimitive} act only on
this reduced form. All trigonometric manipulations are symbolic,
and all coefficient calculations use exact rational arithmetic.

\begin{lstlisting}[style=wlcert,caption={Direct trigonometric recurrence with successive periodicity conditions.},label={lst:weak-recurrence}]
(* Construct phi_0,...,phi_24 for the truncation T_12. *)
nWeak = 12;
(* Reduce products to a sum of harmonics before extracting the mean. *)
trigForm[f_] := Expand[TrigReduce[Expand[f]]];
trigMean[f_] := f /. {Cos[_] -> 0, Sin[_] -> 0};
(* Return the zero-mean primitive of a reduced zero-mean expression. *)
trigPrimitive[f_] := Expand[f /. {
  Cos[t_] :> Sin[t]/D[t, theta],
  Sin[t_] :> -Cos[t]/D[t, theta]}];

phi[0] = 4;
(* Adding delta to phi[j-1] changes the next mean by -delta. *)
constantMultiplier = trigForm[
  2 Sin[theta]^2 (1 - 8 Cos[theta]^2)];
must[trigMean[constantMultiplier] == -1,
  "integration-constant multiplier"];
(* The extra step j = 25 fixes the constant in phi_24. *)
Do[
  (* Form the right-hand side of the order-j coefficient equation. *)
  rhs = trigForm[
    2 Sin[theta]^2 phi[j - 1]
    - 2 Sin[theta]^2 Cos[theta]^2
      Sum[phi[k] phi[j - 1 - k], {k, 0, j - 1}]
    - Sin[theta] Cos[theta] D[phi[j - 1], theta]
    + Sin[theta] Cos[theta]^3
      Sum[phi[k] D[phi[j - 1 - k], theta], {k, 0, j - 1}]];
  (* Choose the preceding integration constant to enforce periodicity. *)
  If[j >= 2,
    delta = trigMean[rhs];
    phi[j - 1] = Expand[phi[j - 1] + delta];
    rhs = Expand[rhs + delta constantMultiplier]
  ];
  must[trigMean[rhs] == 0, "periodicity condition"];
  (* Its integration constant is determined at the next step. *)
  phi[j] = trigPrimitive[rhs];
  must[Expand[D[phi[j], theta] - rhs] == 0,
    "termwise integration"],
  {j, 1, 2 nWeak + 1}
];
(* Compare the first four coefficients with the displayed formulas in the main text. *)
firstPhi = {4, Sin[4 theta] - 2 Sin[2 theta],
  3/8 - 3 Cos[2 theta]/4 + Cos[4 theta]/4
    + 5 Cos[6 theta]/12 - Cos[8 theta]/4,
  Sin[2 theta]/12 + 5 Sin[4 theta]/192
    - 23 Sin[6 theta]/72 + 37 Sin[8 theta]/192
    + 5 Sin[10 theta]/48 - 5 Sin[12 theta]/64};
must[And @@ Table[Expand[phi[j] - firstPhi[[j + 1]]] == 0,
  {j, 0, 3}], "first four coefficients"];
(* Check that even orders are even in theta and odd orders are odd. *)
must[And @@ Table[
  Expand[(phi[j] /. theta -> -theta) - (-1)^j phi[j]] == 0,
  {j, 0, 2 nWeak}], "coefficient parity"];
Print["Trigonometric coefficients determined through order ",
  2 nWeak];
\end{lstlisting}

\paragraph{The substitution $\xi=\cot\theta$ and normalization.}
This is the same variable substitution used in the proof of
Lemma~\ref{lem: n=12, P - Q r'>0}. 
The addition formulas give
\[
 \cos(2k\theta)=\frac{\mathsf C_k(\xi)}{(1+\xi^2)^k},\qquad
 \sin(2k\theta)=\frac{\mathsf S_k(\xi)}{(1+\xi^2)^k},
\]
where the real polynomials $\mathsf C_k,\mathsf S_k$ satisfy
\begin{align*}
 \mathsf C_0&=1,\qquad \mathsf S_0=0,\\
 \mathsf C_{k+1}&=(\xi^2-1)\mathsf C_k-2\xi\mathsf S_k,\qquad
 \mathsf S_{k+1}=2\xi\mathsf C_k+(\xi^2-1)\mathsf S_k.
\end{align*}
The largest harmonic needed for $\varphi_0,\ldots,\varphi_{24}$
is $\cos(96\theta)$ or $\sin(96\theta)$, so the common denominator $(1+\xi^2)^{48}$ suffices. 
Using \eqref{eq: relations between alpha_n beta_n and varphi_n} and \eqref{def: T_n}, 
the code constructs
\[
 \mathcal A(\xi,\mu)
 :=(1+\xi^2)^{48}(1+\mu^2)^{12}T_{12}(\theta,\mu).
\]
Since $\xi\to+\infty$ corresponds to $\theta\to0^+$,
the coefficient of $\xi^{96}$ in $\mathcal A$ is
$D_{12}(\mu)=(1+\mu^2)^{12}T_{12}(0,\mu)$.
We also check this identity directly by evaluating the trigonometric
coefficients at $\theta=0$.

Set $\mathcal A_1:=(1+\xi^2)\partial_\xi\mathcal A-96\xi\mathcal A$.
Since $\mathrm d\xi/\mathrm d\theta=-(1+\xi^2)$, we have
\[
 R_{12}=\frac{4\mathcal A}{(1+\xi^2)^{48}D_{12}},\qquad
 \partial_\theta R_{12}=-\frac{4\mathcal A_1}{(1+\xi^2)^{48}D_{12}}.
\]
Substituting into \eqref{def: definition of V_n} gives
\begin{align*}
 \mu^3\tilde V_{12}={}&
 8 \mu(1+\xi^2)^{49} \mathcal A D_{12}
 -32\mu\xi^2\mathcal A^2
 +4 (1+\xi^2)^{50} \mathcal A_1 D_{12} \\
 &+4\mu\xi (1+\xi^2)^{49} \mathcal A_1 D_{12}
 -16\mu\xi^3\mathcal A_1\mathcal A.
\end{align*}
This normalization is exactly \eqref{def: nomalization of V_n} with $n=12,N_1=98$ and $N_2=3$.
The code variables \texttt{apol},
\texttt{dpol}, and \texttt{ell} represent $\mathcal A$, $D_{12}$,
and $\mathcal A_1$, respectively.

\begin{lstlisting}[style=wlcert,caption={Real trigonometric substitution and exact construction of $\tilde V_{12}$.},label={lst:weak-polynomial}]
(* Here xi = cot(theta); denXi = (1+xi^2)^48 clears denominators. *)
u = 1 + xi^2; denXi = u^(4 nWeak);
(* cosNumer[k]/u^k and sinNumer[k]/u^k represent the harmonics
   cos(2 k theta) and sin(2 k theta), respectively. *)
cosNumer[0] = 1; sinNumer[0] = 0;
Do[
  cosNumer[k + 1] = Expand[
    (xi^2 - 1) cosNumer[k] - 2 xi sinNumer[k]];
  sinNumer[k + 1] = Expand[
    2 xi cosNumer[k] + (xi^2 - 1) sinNumer[k]],
  {k, 0, 4 nWeak - 1}
];
(* Return denXi times phi[j] after the real cotangent substitution. *)
cotNumer[j_Integer] := Module[{ac, bs, constant, restored},
  constant = trigMean[phi[j]];
  ac = Table[Coefficient[phi[j], Cos[2 k theta]], {k, 1, 2 j}];
  bs = Table[Coefficient[phi[j], Sin[2 k theta]], {k, 1, 2 j}];
  must[VectorQ[Join[{constant}, ac, bs], rationalQ],
    "rational trigonometric coefficients"];
  (* Reconstruct phi[j] to check that no harmonics were omitted. *)
  restored = constant + Sum[
    ac[[k]] Cos[2 k theta] + bs[[k]] Sin[2 k theta],
    {k, 1, 2 j}];
  must[Expand[phi[j] - restored] == 0, "harmonic support"];
  Expand[constant denXi + Sum[
    (ac[[k]] cosNumer[k] + bs[[k]] sinNumer[k])
      u^(4 nWeak - k), {k, 1, 2 j}]]
];
Do[f[j] = cotNumer[j], {j, 0, 2 nWeak}];
(* Apply the binomial transforms for alpha_j and beta_j;
   a[j] and b[j] include the common factor denXi. *)
a[0] = f[0]; b[0] = f[1];
Do[a[j] = Sum[Binomial[j - 1, k - 1] f[2 k], {k, 1, j}],
  {j, 1, nWeak}];
Do[b[j] = Sum[Binomial[j - 1, k - 1] f[2 k + 1], {k, 1, j}],
  {j, 1, nWeak - 1}];
(* apol represents A = denXi (1+mu^2)^12 T_12. *)
apol = Expand[
  Sum[a[k] mu^(2 k) (1 + mu^2)^(nWeak - k), {k, 0, nWeak}]
  + mu Sum[b[k] mu^(2 k) (1 + mu^2)^(nWeak - k),
      {k, 0, nWeak - 1}]];
(* The coefficient of xi^96 gives D_12 from the theta -> 0 limit. *)
dpol = Coefficient[apol, xi, 8 nWeak];
(* ell is A_1; dR_12/dtheta = -4 ell/(denXi dpol). *)
ell = Expand[u D[apol, xi] - 8 nWeak xi apol];
(* Form the polynomial numerator mu^3 Vtilde_12. *)
raw = Expand[8 mu apol denXi dpol u - 32 mu apol^2 xi^2
  + 4 ell denXi dpol u^2 + 4 mu xi ell denXi dpol u
  - 16 mu xi^3 ell apol];
(* Verify the exact order of vanishing before dividing by mu^3. *)
must[Take[CoefficientList[raw, mu], 3] == {0, 0, 0},
  "factor mu^3"];
must[Coefficient[raw, mu, 3] =!= 0, "maximal power of mu"];
vtilde = Expand[raw/mu^3];
(* Retain the power coefficient matrix for the Bernstein checks. *)
weakCoefficients = coefficientMatrix[vtilde, {xi, mu}, {194, 48}];
must[{Exponent[vtilde, xi], Exponent[vtilde, mu]} == {194, 48},
  "weak bidegree"];
must[Length[CoefficientRules[vtilde, {xi, mu}]] == 4752,
  "weak monomial count"];
(* Nonnegative even-power coefficients and constant 4 give D_12 >= 4. *)
dCoefficients = CoefficientList[dpol, mu];
must[And @@ Thread[dCoefficients >= 0] && First[dCoefficients] == 4,
  "positive denominator"];
must[Exponent[dpol, mu] == 24 &&
  Expand[dpol - (dpol /. mu -> -mu)] == 0, "even denominator"];
(* On xi=0, verify that Vtilde_12 = mu^22 times a polynomial in mu^2 with strictly positive coefficients; it is positive for mu != 0. *)
axisCoefficients = CoefficientList[vtilde /. xi -> 0, mu];
axisPowers = Flatten[Position[axisCoefficients, Except[0],
  {1}, Heads -> False]] - 1;
must[axisPowers == Range[22, 46, 2] &&
  And @@ Thread[axisCoefficients >= 0], "positive axis polynomial"];
Print["Weak polynomial: ", {194, 48, 4752}];
Print["Denominator coefficients at orders 0, 2, 4: ",
  dCoefficients[[{1, 3, 5}]]];
Print["Exponents of mu with positive coefficients: ", axisPowers];
\end{lstlisting}

The construction yields bidegree $(194,48)$ and $4752$ nonzero monomials.
It verifies the two univariate certificates
\begin{equation*}
 D_{12}(\mu)=\sum_{j=0}^{12}d_j\mu^{2j},\qquad
 \tilde V_{12}(0,\mu)=\mu^{22}\sum_{j=0}^{12}q_j\mu^{2j},
\end{equation*}
where $d_0=4$, all $d_j\ge0$, and all $q_j>0$.
For example, $d_1=1153/24$ and $d_2=18281731/69120$.
The complete exact coefficients are retained in
\texttt{dCoefficients} and \texttt{axisCoefficients}. 
The checks on
their signs also give $D_{12}(\mu)\ge4$ 
and $\tilde V_{12}(0,\mu)>0$ 
for all real $\mu$.

\paragraph{Bernstein verification.}
The matrix \texttt{weakCoefficients} is the $(195\times49)$ power
coefficient matrix of $\tilde V_{12}$.
The reflection $\xi\mapsto-\xi$ multiplies row $i+1$ by $(-1)^i$.
The reciprocal transformation $p(\xi,\mu)\mapsto\xi^{194}p(1/\xi,\mu)$
reverses the row order. Since $194$ is even, these operations commute.
Thus, the four matrices below represent exactly $W_1,W_2,W_3,W_4$
in the proof of Lemma~\ref{lem: n=12, P - Q r'>0}.

\Needspace{24\baselineskip}
\begin{lstlisting}[style=wlcert,caption={Exact Bernstein checks for the four transformed polynomials in the weakly nonlinear regime.},label={lst:weak-checks}]
(* Convert both variables to Bernstein form on the unit square. *)
bx = bernsteinMap[194, {0, 1}];
by = bernsteinMap[48, {0, 1}];
(* Row i corresponds to xi^(i-1); reflection changes its sign factor. *)
negated = Table[(-1)^(i - 1) weakCoefficients[[i]], {i, 1, 195}];
(* Reversing rows implements xi^194 p(1/xi,mu).
   The four matrices correspond to W_1,W_2,W_3,W_4 in that order. *)
weakMatrices = {weakCoefficients, negated,
  Reverse[weakCoefficients], Reverse[negated]};
(* Apply the two variable transforms and collect exact sign reports. *)
weakReports = Table[
  matrixReport[bx . mat . Transpose[by]], {mat, weakMatrices}];
(* Verify the dimensions, sign counts and minima in Table A.1. *)
must[weakReports[[All, 1]] == ConstantArray[{195, 49}, 4],
  "weak matrix dimensions"];
must[weakReports[[All, 2]] == {0, 0, 0, 0},
  "weak Bernstein signs"];
must[weakReports[[All, 3]] == {44, 44, 0, 0},
  "weak Bernstein zero counts"];
must[weakReports[[All, 4]] == {0, 0, 16/3, 3103/582},
  "weak Bernstein minima"];
Print["Weak Bernstein reports: ", weakReports];
\end{lstlisting}

The exact outputs are summarized in Table~\ref{tab:weak-certificate},
whose columns headed Negative, Zero, and Positive count the corresponding  Bernstein coefficients.
The last column of Table~\ref{tab:weak-certificate} gives the minimum
Bernstein coefficient.
Each matrix contains $195\cdot49=9555$ entries. 

\begin{table}[htbp]
\centering
\caption{Exact Bernstein certificates for the weakly nonlinear regime.}
\label{tab:weak-certificate}
\begin{tabular}{c|r|r|r|c}
\hline
Polynomial & Negative & Zero & Positive & Minimum coefficient \\
\hline
$W_1$ & $0$ & $44$ & $9511$ & $0$ \\
$W_2$ & $0$ & $44$ & $9511$ & $0$ \\
$W_3$ & $0$ & $0$ & $9555$ & $16/3$ \\
$W_4$ & $0$ & $0$ & $9555$ & $3103/582$ \\
\hline
\end{tabular}
\end{table}

\subsection{The certificate for the intermediate regime}\label{app:energy}

The required polynomial $\tilde H_{88}(\chi,\nu)$ given in \eqref{eq: tilde H_n} can be evaluated more economically in the following way. 
Set
\[
 h_{k,n}(x,\mu)=x^k\mu^{\varepsilon_k}
 p_{k,n}(\chi,\nu),\qquad
 \chi=x^2,\quad\nu=\mu^2,\quad
 \varepsilon_k=k\bmod2.
\]
For a polynomial $p(\chi,\nu)=\sum p_{ij}\chi^i\nu^j$, define
\[
 \mathcal J_k[p]
 :=\sum_{i,j}\frac{p_{ij}}{k+2i}\chi^i\nu^j,
 \qquad k\ge1.
\]
Direct substitution into \eqref{def: h_{k,n}} with $F(x)=x^3/3-x$ and $g(x)=x$ gives
\begin{align*}
 p_{0,n}&=1,\qquad p_{1,n}=2n(\chi/3-1),\notag\\
 p_{k,n}&=\mathcal J_k\bigl[
 (2n-k+1)\nu^{1-\varepsilon_k}(\chi-1)p_{k-1,n}
 +(2n-k+2)p_{k-2,n}\bigr].
\end{align*}
Indeed, integration of $x^{k-1+2i}$ produces the denominator $k+2i$.
For $n=88$, we have
$\tilde H_{88}=-p_{175,88}$.
In the code for this subsection, \texttt{xi} denotes $\chi=x^2$ and
\texttt{nu} denotes $\nu=\mu^2$.

\begin{lstlisting}[style=wlcert,caption={Exact construction of the reduced energy polynomial.},label={lst:energy-recurrence}]
(* In this subsection xi denotes chi = x^2, and nu denotes mu^2. *)
nEnergy = 88;
(* Implement J_k: divide the coefficient of xi^i nu^j by k+2i. *)
integrateReduced[poly_, k_Integer] := Total[(
  Last[#] xi^#[[1, 1]] nu^#[[1, 2]]/(k + 2 #[[1, 1]]) & /@
  CoefficientRules[Expand[poly], {xi, nu}])];
(* Initialize p_0 and p_1; retain only the two preceding polynomials. *)
previous2 = 1; previous1 = 2 nEnergy (xi/3 - 1);
Do[
  (* nu^Boole[EvenQ[k]] is nu^(1-epsilon_k), epsilon_k = k mod 2. *)
  current = integrateReduced[
    (2 nEnergy - k + 1) nu^Boole[EvenQ[k]] (xi - 1) previous1
      + (2 nEnergy - k + 2) previous2, k];
  previous2 = previous1; previous1 = current,
  {k, 2, 2 nEnergy - 1}
];
(* The final recurrence value is p_175,88, so Htilde_88 = -p_175,88. *)
htilde = Expand[-previous1];
(* Check the bidegree and number of nonzero monomials. *)
must[{Exponent[htilde, xi], Exponent[htilde, nu]} == {175, 87},
  "energy bidegree"];
must[Length[CoefficientRules[htilde, {xi, nu}]] == 7832,
  "energy monomial count"];
(* Store the power coefficients for the rectangle transformations. *)
energyCoefficients = coefficientMatrix[htilde, {xi, nu}, {175, 87}];
Print["Energy polynomial generated: ", {175, 87, 7832}];
\end{lstlisting}

The result has bidegree $(175,87)$ and $7832$ nonzero monomials.
We now apply \eqref{def: bernstein matrix on J} to the four
rectangles $J_i$ in the proof of
Lemma~\ref{lem: g h_2n-1 neq 0 for n=88}, in the same order as there.

\begin{lstlisting}[style=wlcert,caption={Exact Bernstein verification of the four rectangles in the intermediate regime.},label={lst:energy-checks}]
(* Each box is {chi interval, nu interval}; the order is J_1,...,J_4.
   Together they cover [0,4] x [1,144]. *)
domains = {{{2, 4}, {145/2, 144}}, {{0, 2}, {145/2, 144}},
  {{0, 2}, {1, 145/2}}, {{2, 4}, {1, 145/2}}};
(* Transform the same polynomial to Bernstein form on each box. *)
energyReports = Table[
  matrixReport[bernsteinMap[175, box[[1]]] . energyCoefficients .
    Transpose[bernsteinMap[87, box[[2]]]]], {box, domains}];
(* Excluding both negative and zero entries proves strict positivity. *)
must[energyReports[[All, 2]] == {0, 0, 0, 0} &&
  energyReports[[All, 3]] == {0, 0, 0, 0}, "energy positivity"];
(* For each exact minimum m_i, find k_i with 10^k_i <= m_i < 10^(k_i+1).
   These exponents give the lower bounds reported in Table A.2. *)
energyPowers = lowerPower /@ energyReports[[All, 4]];
must[energyPowers == {87, 116, 26, 6}, "energy lower bounds"];
Print["Energy dimensions, negative counts, zero counts: ",
  energyReports[[All, 1 ;; 3]]];
Print["Energy lower-bound exponents: ", energyPowers];
\end{lstlisting}

The exact outputs are summarized in Table~\ref{tab:energy-certificate}.
Each resulting matrix has $176\cdot88=15488$ entries, all strictly
positive. To present the minima compactly, write
$m_k:=\min_{i,j}\bigl(B_{175,87}^{J_k}(\tilde H_{88})\bigr)_{ij}$. The final estimates give
\begin{equation*}
 m_1\ge10^{87},\qquad
 m_2\ge10^{116},\qquad
 m_3\ge10^{26},\qquad
 m_4\ge10^6.
\end{equation*}
The command \texttt{energyReports[[i,4]]} returns the full
exact value of $m_i$.

\begin{table}[!htbp]
\centering
\caption{Exact Bernstein certificates for the intermediate regime.}
\label{tab:energy-certificate}
\begin{tabular}{c|r|r|r|c}
\hline
Rectangle & Negative & Zero & Positive & Lower bound for $m_i$ \\
\hline
$J_1$ & $0$ & $0$ & $15488$ & $10^{87}$ \\
$J_2$ & $0$ & $0$ & $15488$ & $10^{116}$ \\
$J_3$ & $0$ & $0$ & $15488$ & $10^{26}$ \\
$J_4$ & $0$ & $0$ & $15488$ & $10^6$ \\
\hline
\end{tabular}
\end{table}

\FloatBarrier

\bigskip
\paragraph{Declaration of generative AI and AI-assisted technologies in the manuscript preparation process}
During the preparation of this work, the authors used Chatgpt for improving writing and providing design suggestions for the computational algorithms and code routines in the Appendix. 
The authors reviewed and edited the output as needed and take full responsibility for the content of the published article. 

\bibliographystyle{plain}
\bibliography{ref}

\end{document}